\documentclass[onefignum,onetabnum]{siamart190516}
\usepackage{arydshln}
\usepackage{cite}
\usepackage{lipsum}
\usepackage{amsfonts}
\usepackage{graphicx}
\usepackage{epstopdf}
\usepackage{mathtools,bbm,enumerate}
\usepackage[caption=false]{subfig}
\usepackage{algorithmic}
\usepackage{cancel}
\usepackage{comment}
\ifpdf
  \DeclareGraphicsExtensions{.eps,.pdf,.png,.jpg}
\else
  \DeclareGraphicsExtensions{.eps}
\fi
\usepackage{hyperref}
\usepackage{amssymb,amsmath,float,mathabx}
\usepackage[normalem]{ulem}
\newtheorem{conjecture}{Conjecture}

\definecolor{turquoise}{rgb}{0.00,0.40,0.50}    % like "sea green"
\definecolor{violet}{rgb}{1.00,0.00,1.00}	% violet
\definecolor{lightblue}{rgb}{0.00,0.00,0.60}	% light blue
\definecolor{midblue}{rgb}{0.20,0.30,0.80}	% middle blue
\definecolor{orange}{rgb}{1.0,0.50,0.}	% orange
\definecolor{gray}{rgb}{0.2,0.2,0.2}		% middle gray
\definecolor{lightred}{rgb}{0.9,0,0}		% light red
\definecolor{lightgreen}{rgb}{0,0.5,0,0}	% light green
\definecolor{lightturq}{rgb}{0,0.2,0.3}		% light turquoise
\definecolor{lightviolet}{rgb}{0.3,0,0.3}	% light violet

\newcommand{\R}{\mathbb{R}}
\renewcommand{\P}{\mathbb{P}}

\newcommand{\B}{\mathcal{B}}
\newcommand{\Ee}{\mathbb{E}}
\newcommand{\eps}{\varepsilon}

\newcommand{\review}[1]{\textcolor{midblue}{ #1}}
\newcommand{\minor}[1]{\textcolor{orange}{ #1}}

\renewcommand{\review}[1]{{ #1}}
\renewcommand{\minor}[1]{{ #1}}

\headers{Balanced Dynamics in Strongly Coupled Networks}{Qui\~ninao \& Touboul}

\title{Balanced Dynamics in Strongly Coupled Networks\thanks{Submitted to the editors \today.}}

\author{Cristobal Qui\~ninao\thanks{Facultad de Ciencias Biol\'ogicas, Pontificia Universidad Cat\'olica de Chile (\email{cquininao@uc.cl})}
\and Jonathan Touboul\thanks{Department of Mathematics and Volen Centre for Complex Systems, Brandeis University, Waltham MA 
  (\email{jtouboul@brandeis.edu}).}
   }

\usepackage{amsopn}

\usepackage{cite}
\usepackage{lipsum}
\usepackage{amsfonts}
\usepackage{graphicx}
\usepackage{epstopdf}
\usepackage{mathtools,bbm,enumerate}
\usepackage[caption=false]{subfig}
\usepackage{algorithmic}
\ifpdf
  \DeclareGraphicsExtensions{.eps,.pdf,.png,.jpg}
\else
  \DeclareGraphicsExtensions{.eps}
\fi

\graphicspath{{plots/},}

\newsiamremark{remark}{Remark}
\newsiamremark{hypothesis}{Hypothesis}
\crefname{hypothesis}{Hypothesis}{Hypotheses}
\newsiamthm{claim}{Claim}
\newsiamthm{example}{Example}
\usepackage{lineno,hyperref}
\usepackage{amssymb,amsmath,float,mathabx}
\modulolinenumbers[5]
\allowdisplaybreaks
\usepackage{etoolbox} %% <- for \pretocmd, \apptocmd and \patchcmd
\newcommand*\linenomathpatch[1]{%
\cspreto{#1}{\linenomath}%
\cspreto{#1*}{\linenomath}%
\cspreto{end#1}{\endlinenomath}%
\cspreto{end#1*}{\endlinenomath}%
}
\newcommand*\linenomathpatchAMS[1]{%
\cspreto{#1}{\linenomathAMS}%
\cspreto{#1*}{\linenomathAMS}%
\csappto{end#1}{\endlinenomath}%
\csappto{end#1*}{\endlinenomath}%
}

\expandafter\ifx\linenomath\linenomathWithnumbers
\let\linenomathAMS\linenomathWithnumbers
\patchcmd\linenomathAMS{\advance\postdisplaypenalty\linenopenalty}{}{}{}
\else
\let\linenomathAMS\linenomathNonumbers
\fi

\linenomathpatch{equation}
\linenomathpatchAMS{gather}
\linenomathpatchAMS{multline}
\linenomathpatchAMS{align}
\linenomathpatchAMS{alignat}
\linenomathpatchAMS{flalign}
\usepackage[normalem]{ulem}

\usepackage{xr}
\makeatletter
\newcommand*{\addFileDependency}[1]{% argument=file name and extension
  \typeout{(#1)}
  \@addtofilelist{#1}
  \IfFileExists{#1}{}{\typeout{No file #1.}}
}
\makeatother

\usepackage{amsopn}
\begin{document}

\maketitle

\begin{abstract}
\review{Many mathematical models of interacting agents assume that individual interactions scale down in proportion to the network size, ensuring that the combined input received from the network does not diverge. In theoretical neuroscience, Sompolinsky and Van Vreeswijk proposed in 1996 that, should these scalings be violated (and under appropriate conditions), the system may not diverge but rather approach a balanced state where the inputs to each neuron compensate each other (in neuroscience, where inhibitory currents compensate the excitatory ones). We come back to this observation and formulate here a mathematical conjecture for the occurrence of such behaviors in general stochastic systems of interacting agents. From a mathematical viewpoint, this conjecture can be viewed as a double-limit problem in the space of probability measures, which we discuss in detail, as it provides several possible mathematical avenues for proving this result. We provide some numerical and theoretical explorations of the conjecture in classical models of neuronal networks. Moreover, we provide a complete proof of an asymptotic result consistent with one of the double-limit problems in a one-dimensional model with separable coupling inspired by models of chemically-coupled neurons. This proof relies on asymptotic methods, and particularly desingularization techniques used in some PDEs, that we apply here to the mean-field limit of the network as the coupling is made to diverge. From the applications viewpoint, this theory provides an alternative, minimalistic explanation for the widely observed balance of excitation and inhibition in the cerebral cortex not requiring the assumption of the existence of complex regulatory mechanisms. }
\end{abstract}

% REQUIRED
% REQUIRED
\section{Introduction}
\review{A number of classical models in physics and biology rely on the description of macroscopic behaviors emerging from the coordinated activity of a large number of interacting agents, often marked by randomness. Boltzmann's kinetic theory of thermodynamics, which relates the stochastic movements of each atom composing the gas to thermodynamic quantities such as temperature and pressure, constitutes an emblematic example~\cite{boltzmann}. More modern applications include, for instance, models of the stock market in finance, where the decisions of individual investors, whose nature and timing are marked by randomness, result in the emergence of specific stock prices and volatility~\cite{feng2012linking}. \minor{In theoretical ecology, this approach was used, for instance, to describe the collective behaviors emerging from the interaction of individual animals or plants~\cite{berec2002techniques,patterson2020probabilistic}. In these systems, randomness arises in the occurrence and timing of various ecological events that control population sizes, including birth, death, encounters between prey and predators, precipitation, and fires. These phenomena combine to reveal the emergence of specific ecosystems, and some of the core modern theoretical questions studied are the robustness of ecosystems, critical transitions and early warning signals~\cite{hagstrom2023phase,xu2023non,gandhi1998critical}. In neuroscience, these questions arise for the characterization of macroscopic patterns of cortical activity signals emerging from the stochastic dynamics of interacting neurons in brain~\cite{touboul2011noise,renart2004mean,baladron2012mean,deco2008dynamic,sterratt2023principles,fusi2007limits,yu2022metastable,baccelli2021pair}}. From the theoretical viewpoint, the question of finding macroscopic, effective equations to describe large-scale statistical systems has attracted a lot of interest from mathematicians and statistical physicists alike. A variety of techniques have been proposed for deriving macroscopic equations describing the collective behavior of a large system of interacting, stochastic agents. Most have in common the identification of averaging effects resulting from the combination of a large number of random contributions, a \emph{mean-field} of interaction, which is usually a deterministic term that describes the combined impact of the whole system on a given agent. For these approaches to be valid, all require the assumption that individual interactions scale with the network size. \minor{In detail, those techniques assume that the amplitude of pairwise} interactions becomes smaller as the network size increases. \minor{Under these assumptions, in the asymptotic regime where the network size diverges, each neuron interacts with the network through the superposition of a large number of small contributions, which converges (or is approximated)} through a classical limit theorem (law of large numbers, or central limit theorem). These assumptions are often justified by the necessity of avoiding divergence of the interaction term: in the absence of scaling, the interaction term will be on the order of the network size and would diverge in the large network limit. The most common choice for the scaling of the interaction term consists in considering that synaptic connections are inversely proportional to the network size, which results in more-or-less classical versions of McKean-Vlasov theory and limiting behaviors~\cite{touboul2011noise,baladron2011mean,quininao2014limits,touboul2012limits,mischler2015kinetic,robert2014dynamics,touboul2014propagation,fournier2014toy,de2014hydrodynamic,galves2015modeling,carrillo2013classical}.  Another situation that has been widely studied is the case of centered interactions scaled as the inverse of the square root of the network size, where a central limit theorem would lead to the emergence of a complex, non-Markovian equation with a random field of interaction~\cite{sompolinsky1988chaos,arous1995large,arous1997symmetric,cabana1,cabana2}. More recently, instead of rescaling coefficients, Ramanan and collaborators proposed to rescale the probability that agents are connected; in these systems, while interaction coefficients are not small, the combined input remains finite because the number of input sources is finite~\cite{ramanan2023interacting,ramanan2023large,ganguly2024hydrodynamic,lacker2023local,ramanan2022beyond}, and there again, mean-field behaviors emerge.  }

About 30 years ago, a celebrated and pioneering paper by Haim Sompolinsky and Carl van Vreeswijk~\cite{sompolinsky-van-vreeswijk} \minor{suggested that, in a system where the scaling of the interaction term is insufficient to ensure its asymptotic boundedness, the system may adjust itself to compensate for the divergence. In detail, Sompolinsky and van Vreeswijk's paper investigated the dynamics of a neural network system composed of two neural populations: excitatory cells, which, when they activate, tend to activate other cells, and inhibitory neurons that, when they activate, tend to decrease the activity of the cells they are connected to. In such a system, Sompolinsky and van Vreeswijk suggested that, in the absence of mean-field rescaling, the network activity would quickly evolve to reach a state where large excitatory inputs to any given neuron would be perfectly counterbalanced by an equally large inhibitory input. This regime where excitation and inhibition balance each other is referred to as the \emph{balanced regime}. Its properties are not well described by the type of mean-field dynamics discussed above}.  \review{The importance of this observation of the convergence of networks to a balanced regime is hard to overstate. \minor{First, instead of mean, deterministic or stochastic behaviors, the networks introduced in~\cite{sompolinsky-van-vreeswijk} typically display chaotic activity, which was found comparable to some irregular brain activity patterns observed experimentally in various studies~\cite{Britten:93, Shadlen:98, Compte:03}. From the functional point of view, multiple studies found that the type of chaotic activity emerging from balanced networks supports rich coding capabilities~\cite{sussillo:09, monteforte:12,london:10,kadmon2020predictive,deneve2016efficient}. Moreover, the study of the dynamics of balanced networks opened} a new area of research with deep, far-reaching ramifications, from statistical physics to stochastic processes and measure theory, to information theory and computer science}~\cite{cabana1,cabana2,jahnke2009chaotic,kadmon2020predictive,terada2024chaotic,wainrib2013topological}. In sharp contrast with this abundance of works exploring the complex behaviors of chaotic neuronal dynamics regimes emerging and their computational capabilities, the remark made \emph{en passant} in that article regarding to the rapid convergence towards balanced states has remained largely underexplored. 

\minor{From the applications viewpoint, identifying which scaling is more consistent with the systems considered would be essential to assess the relevance of the Sompolinsky and van Vreeswijk regime. Typically, these scalings are introduced formally, and experimental studies estimating the amplitude of the interactions as a function of the number of interacting agents are scarce. A notable exception is the work of Barral and Reyes~\cite{barral2016synaptic} that directly estimated in neural networks the relationship between the strength of interaction and the number of inputs received by neurons. This experimental \emph{tour de force} was made possible thanks to their ability to observe dissociated cells generate networks in vitro. In their experimental preparation, they observed the spontaneous development of axons and dendrites from initially dissociated cells, spontaneously yielding a network. From these preparations, the authors were able to measure, for each neuron, the number $n$ of pre-synaptic neurons (that is, the number of neurons sending an input to that particular neuron), as well as the strength of each connection $J$. Strikingly, they found that $J$ was not scaling as $1/n$, as most mean-field models assume, but instead as $1/\sqrt{n}$, consistently with the scaling considered by Sompolinsky and Van Vreeswijk and hypothesized to ensure balance of excitation and inhibition, and generate irregular activity (two properties also studied by Barral and Reyes in~\cite{barral2016synaptic}). This observation further highlights the importance of understanding the dynamics of statistical systems in this regime. }

\review{Mathematically, this is certainly an interesting problem to investigate. Indeed, in addition to possible averaging effects arising in more classical asymptotic problems in large-scale networks, one also needs to handle a possible rapid shift of the network activity to a balanced state. Mathematical results in this direction are scarce. In a somewhat related work, we showed that mean-field limits of neuronal networks with diverging coupling in the context of neural networks with electrical synapses lead to clamping of the solutions onto a Dirac mass where the interaction term vanished~\cite{quininao2020clamping}. While not described as such in that manuscript, this result provided a mathematical proof for a particular type of convergence to a balanced regime. However, the particular form of coupling used in that paper, diffusive coupling, did not differentiate between excitation and inhibition, and instead forced neurons to agree on a common voltage, leading to the total input effectively canceling out. In this sense, the network achieves a form of balance in the input each neuron receives. From a mathematical point of view, this particular choice of coupling was also instrumental in deriving analytical results, and generalizing the approach to different types of coupling, with excitatory and inhibitory neural populations, is a highly nontrivial task. In~\cite{quininao2020clamping}, the proof of convergence to a balanced regime focused on the analysis of a mean-field equation where the coupling coefficient was made to diverge. The sequence of probability measures was then studied using desingularization techniques (Hopf-Cole methods) under the assumption that initial conditions showed a Gaussian concentration to a Dirac mass. This blow-up method allowed showing that the sequence of measures collapsed to a Dirac mass, and that the blow-up profile was Gaussian. Following this work, an alternative approach was proposed, using Wasserstein distances, and used to lift the assumption on initial conditions~\cite{blaustein2023concentration} or achieve a stronger notion of convergence~\cite{blaustein2023large}, where interactions were extended to include a dependence in space. All methods used in these papers strongly rely on the particular choice of connectivity made, and all relied on the study of sequences of measures describing a mean-field limit of the system. Here, we go beyond the particular case and formulate this problem in mathematical terms by proposing that this observation is one facet of a general mathematical conjecture on asymptotic regimes of strongly connected networks.}

\review{The rest of the manuscript is organized as follows. Section~\ref{sec:math} introduces a general setting for the study, and formulates our conjecture. Section~\ref{sec:examples} provides two applications of the conjecture to neurosciences, and provides numerical simulations of these regimes in a few different settings. Section~\ref{sec:viewpoints} discusses different possible angles to approach the conjecture theoretically, by introducing it as a double-limit problem. This yields two distinct approaches. The first limit consists of considering what happens at very short timescales (typically, at timescales inversely proportional to the rate of divergence of the interaction term). We suggest that a non-stochastic dynamical system governs the dynamics at these timescales, with all variables fixed except those affected by the divergence of the interaction. From this perspective, the balanced regime may naturally appear as a dynamical equilibrium of the system on short timescales. The other limit involves first considering that the network size diverges (under appropriate scaling assumptions on the coupling coefficients) and then taking the interaction term to infinity, thereby prioritizing the accounting for collective effects before considering the emergence of balance. The latter framework is formalized in section~\ref{sec:theory}, which provides a complete proof supporting the result stated in the conjecture. The proof relies on functional analysis techniques and addresses a particular system given by a one-dimensional equation with separable coupling. The details of the proof are provided in the Appendix. We conclude by reviewing the main challenges and the relevance of this regime, particularly its applications to neurosciences. }

\section{Mathematical framework and conjecture}\label{sec:math}
In the original work of Sompolinsky and van Vreeswijk~\cite{sompolinsky-van-vreeswijk}, the authors considered a system composed of two populations: \review{excitatory neurons that, when they activate, tend to increase the activity of other neurons, and inhibitory neurons that, when they activate, reduce the activity of other cells. Each of these populations consists of a large number of neurons, and these neurons interact through a random, sparse network. Making the problem slightly more general, we define here an abstract model of $N$ stochastic interacting agents, distributed into $P$ different populations labeled $\alpha\in\{1\cdots,P\}$. The type of interactions between neurons depends on the population to which they belong (these generalize the excitatory/inhibitory populations described in the original work of Sompolinsky and van Vreeswijk~\cite{sompolinsky-van-vreeswijk}). Each agent's activity is defined by a state variable, $(x^i_t)_{t\geq 0}$, which is a stochastic process taking values in $\R^d$. This multi-dimensional vector describes the dynamics of each agent, and generalizes the typical variables that describe a neuron's activity. Typically, in neuroscience, the quantities used to describe the state of a neuron are its firing rate or, in more biophysically realistic descriptions, a vector that includes the electrical membrane potential, ionic currents, and recovery variables. These systems are described by a general stochastic network equation of the type:}
\begin{equation}\label{eq:general_network}
dx^i_t = \left(f_{p_i}(x^i_t) + \sum_{j=1}^N J^N_{ij} \,b_{p_i,p_j}(x^i,x^j) \right)\,dt+ \sigma_{p_i} dW^i_t. 
\end{equation}
where $p_i$ denotes the population of agent $i$, $(x^i_t)_{t\geq 0} \in C(\R,\R^d)$ is the state of agent $i$ in the network, $f_p:\R^d\mapsto \R^d$ represents the intrinsic dynamics of agents in population $p$, and $b_{p,q}(x,y):\R^d\times \R^d\mapsto \R^d$ represents the impact of an agent in population $q$ and in state $y$ on an agent in population $p$ and state $x$. \review{The stochastic processes $(W^i_t)_{t\geq 0}$ are independent $K$-dimensional Brownian motions describing $K$ independent sources of noise, while the parameters $\sigma_{p} \in \R^{d\times K}$ encode the contributions to each noise source onto each dimension of the state variable. In these equations, $J_{ij}^N$ is an $N\times N$ matrix whose coefficients are the coupling intensities between neurons $i$ and $j$, whose scaling is assumed to depend on the network size.}

\review{We are interested here in the behavior of the system in the large $N$ limit. For specificity and simplicity of notation, we assume in the rest of the manuscript that the populations contain the same number of neurons, noted $n$ (so $N=Pn$)\footnote{Most statements remain true when assuming that populations contain distinct numbers of neurons, as long as the number of neurons in each population diverges. In that case, the largest population will be the relevant large parameter controlling the speed of convergence.}, and the coupling coefficients $J_{ij}^N$ (or their statistics, when they are random) only depend on the populations of neurons $i$ and $j$ and the network size $N$. The choice of the coupling coefficient and its scaling with the network size is crucial to determining the asymptotic behavior of the network.}

\review{When $J_{ij}^N=\frac 1 n g_{p,q}$ where $(p,q)$ are the population labels of neurons $(i,j)$ and under} appropriate assumptions on the regularity and decay of the functions $f_p$ and $b_{p,q}$, classical mean-field theory implies that, for all neuron $i$ in population $p\in \{1,\cdots,P\}$, the process $(x^i_t)_{t\geq 0}$ converges in law towards $(\bar{x}^p_t)_{t\geq 0}$ the solution of the system of non-Markovian McKean-Vlasov equations:
\begin{equation}\label{eq:nonlinear_process}
d\bar{x}_t^{p} = \left(f_p(\bar x_t^p) + \sum_{q=1}^P g^{p,q}\,\int_{y} b_{p,q}(\bar x_t^p,y)\mu_t^q(dy) \right)\,dt+ \sigma_p dW^p_t
\end{equation}
where $\mu_t^q$ is the implicitly defined probability distribution of $\bar{x}^q(t)$. 
\review{Another situation studied in the literature previously relates to randomly connected networks, where the $J_{ij}^{N}$ are iid, centered coefficients scaling as $1/\sqrt{n}$. In this situation, the network dynamics converge to a complex, non-Markovian equation, similar to the original system, but where the interaction term is replaced by a Gaussian process whose covariance depends on the solution of the equation. This result is evocative of a functional Central Limit Theorem, but the proof is vastly different and requires somewhat more sophisticated stochastic analysis tools, including, for instance, large-deviation theory~\cite{sompolinsky1988chaos,arous1995large,arous1997symmetric,cabana1,cabana2}.}

\review{In both of these cases, the sum of all coefficients is (almost surely) bounded, and therefore no divergence of the coupling term happens. These dynamics agree, heuristically, with Boltzmann's original molecular chaos hypothesis, also known as sto\ss zahlansatz (often referred to mathematically as the propagation of chaos property). In the kinetic theory of gases, the sto\ss zahlansatz posits that the velocities of colliding gas particles are uncorrelated, and independent of their positions. In the more general context of this paper, it consists in considering that the states of the agents are independent, identically distributed, and thus allow applying one of the classical limit theorems of probabilities (law of large numbers or central limit theorem) to provide accurate predictions of the asymptotic behavior of the system.}

\review{Here, we consider cases where the sum of the coefficients diverges. \minor{Beyond its mathematical interest, it is, as reviewed above, relevant in computational neuroscience, the context of Sompolinsky and van Vreeswijk's claim in~\cite{sompolinsky-van-vreeswijk}, and to account for the coupling regime identified by Barral and Reyes in~\cite{barral2016synaptic}.}} In that case, Boltzmann's stochastic number is not effective in providing predictions for the system's dynamics in the large $n$ limit, since, under this assumption, the interaction term has a positive probability of diverging. To simplify our presentation further, let us assume that the divergence rate is uniform for any pair of populations, denoted $\gamma(n)$ for some function $\gamma$ that tends to infinity as $n\to \infty$. We note $J_{ij}^N=\gamma(n)\frac{g_{p,q}}{n}$, so the equation for neuron $i$ in population now reads:
\begin{equation}\label{eq:gamma_network}
dx^i_t = \left(f_{p}(x^i_t)+\frac{\gamma(n)}{n} \sum_{q=1}^P  g_{p,q} \sum_{j;p_j=q} \,b_{p,q}(x^i,x^j) \right)\,dt+ \sigma_{p} dW^i_t. 
\end{equation}
Building upon Sompolinsky and Van Vreeswijk's intuition, we propose that instead of a convergence of the system to an unconstrained equation, the divergence of the interaction term will govern the behavior of the system at short timescales, and will lead to two possible behaviors: the dynamics may run away and variables may quickly increase or decrease to reach large positive or negative values, or, if a balanced regime exists and is attractive, the system will quickly collapse onto the set of balanced dynamics. This balanced regime is defined as the set of voltage distributions that correspond to a zero net input. In detail, we define the \emph{balanced regime} as the set of measures such that, for any point on the support 
\begin{equation}\label{eq:balance_manifold}
\B=\{(\mu^1,\cdots \mu^P)\in \mathcal{M}^1(\R^d)^P,\,\, \forall x_p \in Supp(\mu^p), \sum_{q=1}^P g_{p,q} \int_{\R^d} b_{p,q}(x_p,y)\mu^q({dy})=0\},
\end{equation}
 where the distribution of the random variable is such that the net input received is zero. This regime may be composed of a continuum of solutions on which the system may pursue its dynamics while remaining confined to the manifold.
 
\review{We note that, for any possible measure $\mu$ in the balanced regime, $\B$ is invariant under the differential equations:
\begin{equation}\label{eq:earlyBehaviors}
\frac{dx_p}{dt}=\sum_{q=1}^{P} g_{p,q}\int_{\R^d} b_{p,q}(x_p,y)\mu_q(\review{dy}),
\end{equation}
in the sense that if $(\mu_1,\cdots,\mu_P)\in\B$, then for any $x_p$ in the support of $\mu^p$, $x_p$ is a fixed point of~\eqref{eq:earlyBehaviors}. This dynamical system will be used to characterize the stability of the balanced regime. 
} 

In detail, we propose the following conjecture. 

\begin{conjecture}\label{conjecture}
\review{Under the current formalism, let $X^n=(x^n_{i_1},\cdots, x^{n}_{i_p})$ be a vector composed of one representative neuron in each of the $P$ populations. 
If the set $\B$ is not empty, then we conjecture that there exist asymptotic solutions that are confined to the balanced regime for all times. }

\review{Moreover, if the balance manifold $\B$ given by~\eqref{eq:balance_manifold} \review{is} \emph{attractive} for the dynamical system~\eqref{eq:earlyBehaviors}, in the sense that for initial conditions $x_p^0$ near the support of $\mu_p$, the solutions of equation~\eqref{eq:earlyBehaviors} converge to the support of $\mu_p$, then even for initial conditions outside of $\B$ the system will collapse onto $\B$ at a timescale inversely proportional to $\gamma(n)$ (that is,  for any $t>0$, $\P[X^n_t\in \B]\to 1$ as $n\to \infty$). }
\end{conjecture}

\review{While broader and more abstract, this conjecture, however, reflects precisely the balanced regime predicted by Sompolinsky and van Vreeswijk in~\cite{sompolinsky-van-vreeswijk}, since the set $\B$ is composed of distributions associated with regimes where the input received by each neuron from the rest of the network combines to a zero net input.} Such regimes do not always exist. For instance, in a one-dimensional system and in the case where $b(x,y)$ is of constant sign, no measure can ensure balance. However, if there exists a solution $(x_1,\cdots,x_P)\in (\R^d)^P$ to the system of equations:
\[\sum_{q=1}^{P} g_{p,q}  b_{p,q}(x_p,x_q) \review{= 0}, \qquad p\in \{1\cdots P\},\]
then the Dirac mass at this point belongs to the balanced regime.

\review{Before we proceed to exploring the mathematical underpinnings of this conjecture, we first examine its predictions for various neural network models. }

\section{Asymptotic behaviors of electrically or chemically coupled networks of FitzHugh-Nagumo neurons}\label{sec:examples}
\review{The FitzHugh-Nagumo model is a phenomenological mathematical neuron model introduced in the middle of the past century~\cite{fitzhugh1955mathematical,nagumo1962active}. It is particularly popular in computational neuroscience for its relative simplicity and its ability to capture critical characteristic features of neuronal dynamics, including, in particular, neural excitability~\cite{cebrian2024six}. The FitzHugh-Nagumo model describes the coupled dynamics of a voltage variable $x$ and a recovery (or adaptation) variable $y$, subject to the differential equation:
\[
\begin{cases}
\dot{x}=f(x)-y+I\\
\dot{y}=a(b\,x-y)\\
\end{cases}
\]
where $f$ is a cubic polynomial with negative leading coefficient (typically, $f(v)=v(1-v)(v-\lambda)$), $I$ represents the input to the neuron, $a$ is the (typically small) timescale ratio between the recovery variable and the voltage, and $b$ is a positive parameter controlling the response of the recovery variable (its `adaptation') to changes in the voltage. The network models described below build upon the FitzHugh-Nagumo equation and explore the dynamics of large networks with two choices of coupling.}

\subsection{Electrically coupled FitzHugh-Nagumo equation}\label{sec:ElectricalExample}
\review{As a first example, we consider a network composed of a single population ($P=1$) of electrically coupled FitzHugh-Nagumo neurons~\cite{quininao2020clamping,mischler2015kinetic}:}
\[\begin{cases}
dx^i_t=\Big(f(x^i_t)-y^i_t+\frac{\gamma(n)}{n}\,\minor{g}\, \sum_{j=1}^n (x^j-x^i)\Big)\,dt + \sigma dW^i_t\\
dy^i_t=a(bx^i_t-y^i_t)\review{\,dt}.
\end{cases}\]
\review{where $(x^i,y^i)$ are the voltage and recovery variables of neuron $i$, $g>0$ is a constant coupling strength and the processes $(W^i_t)_{t\geq 0}$ are independent, scalar Brownian motions}. This model fits in the main setting of the problem with $d=2$, $P=1$, and $b(x,y)=y-x$. \review{It thus provides an elementary example to break down the conjecture. In this model:}
\begin{itemize}
\item \review{the \emph{balance manifold} is composed of all probability measures $\mu$ on $\R^2$ such that, for any $(x,y)$ in their support, 
\[\int_{\R^2} (x-x')\mu(\review{dx',dy'})=x-\int_{\R^2} x'\mu(\review{dx',dy'})=0.\] 
In other words, a measure in the balance manifold will have all of the voltage variables in its support coincide (and equal to the average voltage of the distribution). The balance condition provides no other constraints on the distribution of adaptation variables. The balance manifold is thus given by:
\[\B=\{\mu(\review{dx,dy})=\delta_{x^\star}(dx)\times \nu(dy), \text{for } x^{\star} \in \R \text{ and } \nu\in \mathcal{M}^1(\R)\} \]
where $ \mathcal{M}^1(\R)$ are the probability distributions on $\R$. }
\item \review{Let us now consider $\mu=\delta_{x^\star}(dx)\otimes \nu(dy)\in \B$ an arbitrary measure on the balance manifold and let us characterize its stability. } Equation~\eqref{eq:earlyBehaviors}, in this context, reduces to:
\begin{equation}
\begin{cases}
\frac{dx}{dt} = g \int_{\R^2} (x'-x)\,\delta_{x^\star}(dx')\otimes \nu(dy')= x^\star-x,\\
\frac{dy}{dt} =0
\end{cases}
\end{equation}
which is a simple linear equation that, for each initial condition, converges to $x^\star$, which is precisely the support of the distribution. Therefore, it is stable in the sense of the conjecture, since the solution converges toward the support of the distribution. 
\end{itemize}
\review{These results are comparable to the findings of~\cite{quininao2020clamping}, where the phenomenon of mass concentration at Dirac distributions was referred to as clamping and synchronization, depending on whether the slower evolution of the system on the balance manifold was stationary or dynamic. }

Figure~\ref{fig:FhNElectrical} illustrates this convergence for a particular set of parameters and for $\gamma(n)=n$ or $\sqrt{n}$. In both cases, we observe a quick convergence to a balanced regime illustrated by the rapid convergence of all voltages to a common value (Figure~\ref{fig:FhNElectrical} A,A' and insert) associated with a quick convergence of the distribution of voltages to a Dirac mass (Figure~\ref{fig:FhNElectrical}C,C') and a rapid collapse of the standard deviation of the voltage variables, \review{thereby illustrating the convergence result stated in the conjecture. Moreover, these simulations illustrate that while both choices of scaling converge to the same dynamics, when the divergence of the coupling term is slower ($\gamma(n)=\sqrt{n}$), the convergence is also slower (in time) and less dramatic (larger standard deviation). This is why we chose a relatively small network size, $N=300$, since for larger values of $n$, both the speed of convergence and the steady standard deviation decrease further, for both scalings, making differences more subtle to identify.} 

\review{Of note, after this rapid convergence phase, the system, quickly confined on balanced regimes, starts wandering on the manifold, either converging to a fixed voltage value or displaying non-trivial dynamics such as the relaxation cycles shown below. The determination of these eventual behaviors is beyond the scope of the conjecture, which only discusses the existence and stability of solutions confined to the balance manifold, but not their dynamics. To date, there is no systematic approach for uncovering these dynamics on the balance manifold. In~\cite {quininao2020clamping}, moment approximations under the assumption of rapid convergence were developed to infer the dynamics of the clamping voltage $x^\star$ and its associated adaptation variable. }

\begin{figure}[h]
\includegraphics[width=\textwidth]{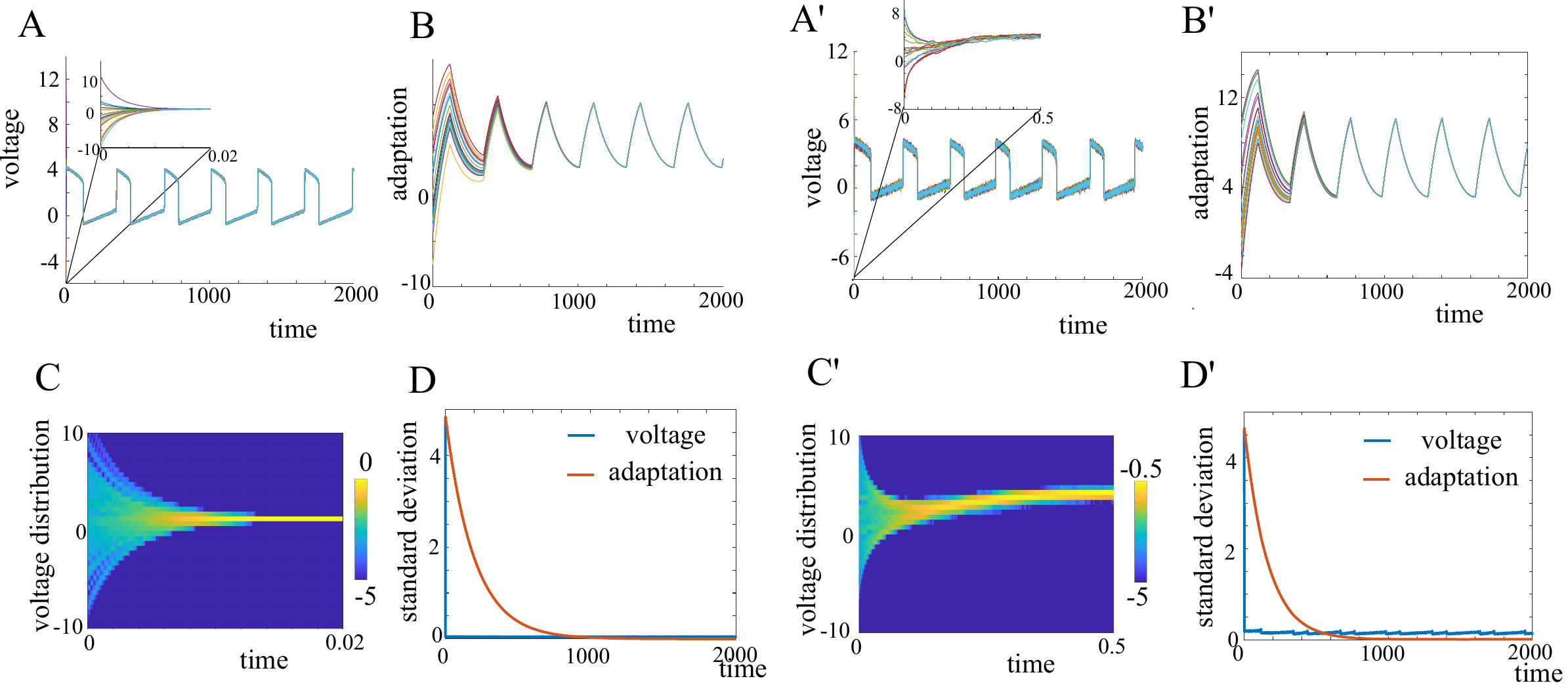}
\caption{\review{Dynamics of the electrically coupled FitzHugh Nagumo with $n=300$ neurons with $\gamma(n)=n$ (A,B,C,D) or $\gamma(n)=\sqrt{n}$ (A',B',C',D'). (A,A') voltage variable for 20 neurons in the network (with inset showing the rapid convergence to balance), (B,B') associated recovery variable, (C,C') represent the distribution of the voltage (in logarithmic scale) at small timescales and (D,D') the standard deviation of both variables as a function of time.} Quick convergence to the balanced regime is visible in the collapse of the distribution and of the standard deviation in the voltage variable. Simulations were performed using a custom Euler-Maruyama code on Matlab. $f(v)=v(1-v)(v-4)+4$,  $a=0.005$, $b= 6$, $g=1$ and $\sigma=1$, and independent Gaussian initial condition with standard deviation equal to $5$, mean initial voltage $1$ and mean initial recovery variable $1.5$.}
\label{fig:FhNElectrical}
\end{figure}

\subsection{Example 2: FitzHugh-Nagumo model with chemical synapses and excitatory and inhibitory populations}\label{sec:Chemical}
\review{In this section, we consider the dynamics of a two-population neural network with $n$ excitatory and $n$ inhibitory neurons (thus, $N=2n$), with interactions modeled as chemical synapses. This type of synapse, the most ubiquitous in the cerebral cortex, mediates the communication between cells through the release of vesicles filled with neurotransmitter molecules that bind to specific receptors on the membrane of the post-synaptic cell. This binding modifies the neuron's conductances, leading to changes in ion exchange across the membrane and generating a current that alters the cell's voltage. These vesicles are released when the neuron fires an action potential (or spike), which corresponds to a substantial increase in the cell membrane's voltage. The subsequent modification of the membrane's conductance features a rapid rise and a slower decay, which was modeled in~\cite{destexhe1998kinetic} as a first-order differential equation nonlinearly coupled to the voltage (through a sigmoidal function of the voltage $\alpha(x)$) to reflect the spike-induced change in conductances. This change in conductance in turn affects the voltage by generating an ion transfer proportional to the difference between the membrane potential and a constant voltage, called the reversal potential, and characteristic of the excitatory ($E_E$) or inhibitory ($E_I$) synapse:}
		\begin{equation}\label{eq:EIFhN}
		\begin{cases}
			dx^i_t&=\displaystyle{\Big(f(x^i_t)-y^i_t+\frac{\gamma(n)}{n}\sum_{j=1}^{N}g_{p(j)p(i)} (x^i_t-E_{p(j)}) s^{j}_t\Big)\,dt +\sigma dW^i_t}\\
			dy^i_t&=a\Big(bx^i_t-y^i_t+c\Big)\,dt\\
			ds^{i}_t&=\Big(-\frac{s^{i}_t}{\tau}+\alpha(x^i_t)(1-s^{i}_t)\Big)\,dt
		\end{cases}
	\end{equation}
\review{and $s^i$ are the synaptic conductances associated with neuron $i$\footnote{We assumed here, for mathematical simplicity, that the dynamics of excitatory and inhibitory cells are identical. Similarly, more biologically realistic developments should differentiate between timescales and dynamical parameters of each population. Applying conjecture~\ref{conjecture} to this model yields the following predictions:}}
 \begin{itemize}
 \item The balance regime is defined by the set of measures $(\mu_E,\mu_I)$ on $\R^3$ such that any voltage variable $x_E,x_I$ in the support of the distributions,
	\[\begin{cases}
	g_{E E} (x_E-E_{E})\bar{s}_E + g_{I E} (x_E-E_{I})\bar{s}_I=0\\
	g_{E I} (x_I-E_{E})\bar{s}_E + g_{I I} (x_I-E_{I})\bar{s}_I=0
	\end{cases}\]
	where $\bar{s}_E=\int_{\R^3} s \mu_E(dx,dy,ds)$ and $\bar{s}_I= \int_{\R^3} s\mu_I(dx,dy,ds)$. These equations define unique voltages (provided that the denominators do not vanish):
	\begin{equation}\label{eq:2popsBalanceVoltage}
	\begin{cases}
	\displaystyle x_E= x_E^\star:=\frac{g_{EE}\, E_E \, \bar{s}_E\,+\,g_{IE}\,E_I \, \bar{s}_I}{g_{EE} \, \bar{s}_E+g_{IE}\, \bar{s}_I}\\[10pt]
	\displaystyle x_I= x_I^\star:=\frac{g_{EI} \, E_E \, \bar{s}_E+g_{II} \, E_I \, \bar{s}_I}{g_{EI} \, \bar{s}_E+g_{II} \,\bar{s}_I},
	\end{cases}
	\end{equation}
	\review{and impose no constraints on other variables. The balanced regime is thus given by the set of measures:
	\begin{multline}
	\B=\Big\{\mu=\Big (\delta_{x_E^\star(\nu_E,\nu_I)} (dx)\; \otimes \; \nu_E(dy,ds)\; ,\; \delta_{x_I^\star(\nu_E,\nu_I)}(dx) \; \otimes \; \nu_I(dy,ds)\Big), \\
	\quad \nu_E,\nu_I\in \mathcal{M}^1_+(\R^2)\Big\}
	\end{multline}
	where $(x_E^\star,x_I^\star)$ are coupled to $\nu_E$ and $\nu_I$, given by equations~\eqref{eq:2popsBalanceVoltage} with $\bar{s}_p=\int_{\R^2} s \; \nu_p(dy,ds)$ for $p\in \{E,I\}$. }
	
	\item \review{In terms of stability, the conjecture predicts that a particular measure $\mu=(\delta_{x_E^\star}(dx)\otimes \nu_E(dy,ds),\delta_{x_I^\star}(dx)\otimes \nu_I(dy,ds))\in \B$ will be stable if, for any initial condition in the vicinity of the support of the measure converges to the support under the deterministic dynamical system~\eqref{eq:earlyBehaviors}. Here, the support of $\mu$ for the voltage variables are the singletons $\{x_E\}$ and $\{x_I\}$, and equation~\eqref{eq:earlyBehaviors} are given by:
	\begin{equation}
	\begin{cases}
		\dot{x}_E= (g_{EE} \bar{s}_E+g_{IE} \bar{s}_I)x_E - (g_{EE}E_E\bar{s}_E + g_{IE}E_I\bar{s}_I),\\
		\dot{x}_I = (g_{EI} \bar{s}_E+g_{II} \bar{s}_I)x_I - (g_{EI}E_E\bar{s}_E + g_{II}E_I\bar{s}_I),
		\end{cases}
	\end{equation}
	where the values of $y_E, \bar{s}_E, y_I, \bar{s}_I$ are constant and $\bar{s}_{E,I}$ are the expectations of the variable $s$ for the E/I population. The balanced measure will thus be stable when $G_{EE} \bar{s}_E+G_{IE} \bar{s}_I<0$ and $G_{EI} \bar{s}_E+G_{II} \bar{s}_I<0$. Because the excitatory weights are positive and the inhibitory weights negative, and the synaptic variables $\bar{s}_{E,I}>0$, this indicates that the balance manifold will be stable in \emph{inhibition-dominated} regimes, a condition for stability that has been emphasized many times, in other contexts, in neuroscience~\cite{tartaglia2017bistability,sadeh2021inhibitory,ahmadian2021dynamical}. }
\end{itemize}
%In conclusion, the voltages of the neuron all converge to $(x_E^{\star},x_I^{\star})$ (thus, according to the conjecture, the balance regime is stable) when $G_{E\beta} \bar{s}_E+G_{I\beta} \bar{s}_I<0$ for $\beta\in \{E,I\}$. 
Figure~\ref{fig:FhNChemical} shows simulations of the above system for particular choices of parameters. In Figure~\ref{fig:FhNChemical}(A), we represent a case where the inhibitory synaptic coefficients dominate, so that the stability condition is more easily satisfied. Simulations showed a rapid collapse of the voltages on their respective balance voltages $(x_E^{\star},x_I^{\star})$ where they stabilized. To illustrate the dynamic nature of the balance regime, we simulated a sudden change in synaptic conductances (greyed region). We observed that the network quickly adjusts to the new balance voltages, confirming that the balanced regime is stable to perturbations and changes in, for example, electrophysiological conditions or stimulation. Figure~\ref{fig:FhNChemical}(B) instead shows a case where excitatory conductances are larger than inhibitory conductances. In that case, we observe that the system does not converge to the balanced voltage; neurons' voltages quickly escape the vicinity of the balanced regime, clustering into two groups that reach either high voltages or low voltages, thereby deviating from the balanced regime. The resulting current generated by the final double-Dirac measure obtained was found to be fairly large, and increased with $n$. The stabilization, occurring at finite $n$ thanks to the presence of the cubic nonlinearity, is not expected to persist in the limit $n\to\infty$. 

\begin{figure}[h]
\centerline{\includegraphics[width=.8\textwidth]{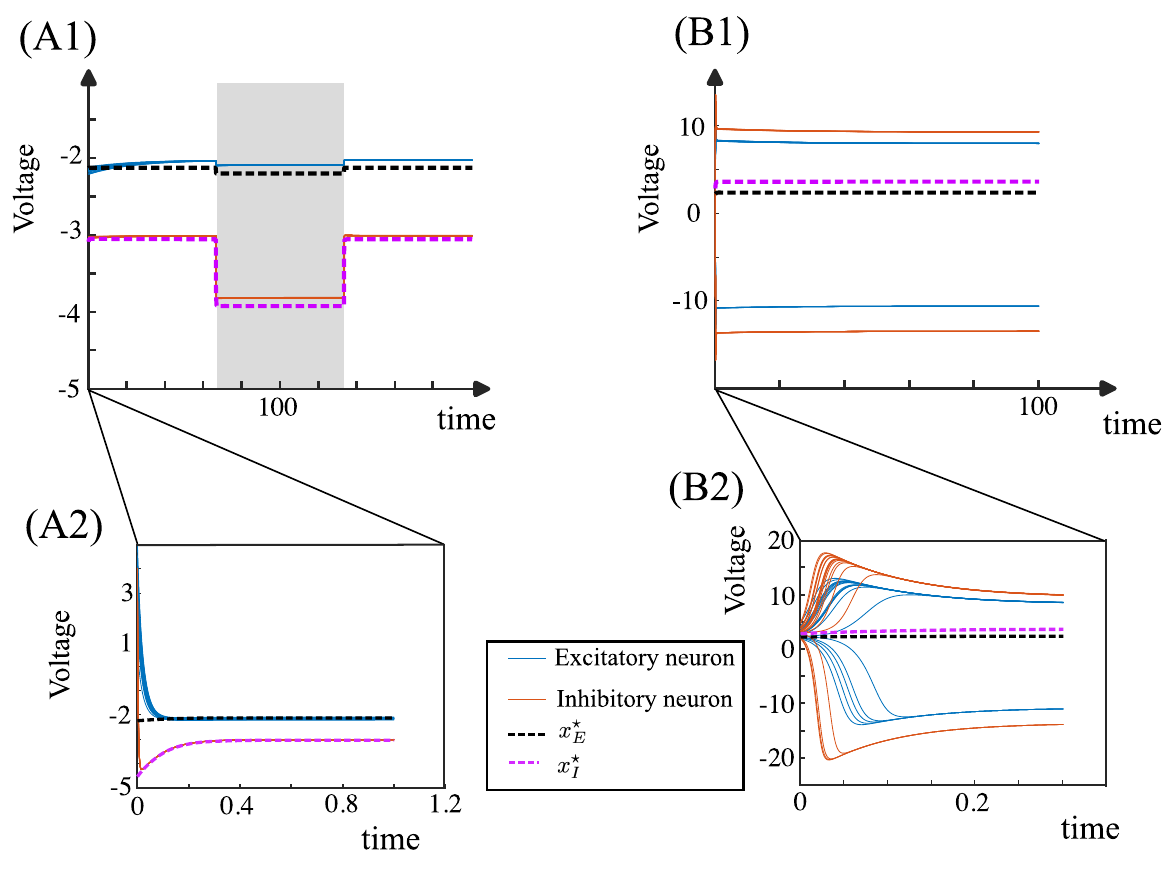}}
\caption{Dynamics of the chemically coupled Fitzhugh Nagumo model~\eqref{eq:EIFhN}. 
\review{(A1-2) Inhibition-dominated case with a stable balanced state ($g_{EE}=0.3$, $g_{EI}=2$, $g_{IE}=1$ and $g_{II}=10$), and (B1-2) Excitation-dominated regimes, where we see no convergence to a balanced state ($g_{EE}=1$, $g_{EI}=2$, $g_{IE}=0.1$ and $g_{II}=0.7$). (A2) and (B2) represent the dynamics of the system at short timescales for 20 excitatory (blue) and 20 inhibitory (red) neurons, together with the predicted value for the balanced excitatory (dashed black) and inhibitory (dashed purple) voltage. A clear, rapid convergence to the balanced state occurs in (A2), whereas in (B2), the voltages of neurons in each population escape the balanced regimes. Instead, they cluster into two groups: those converging to a common value above the balanced voltage and the rest converging to a common value below the balanced voltage. At longer timescales, those regimes remain stable, as shown in (A1) and (B1). (A1, greyed region) solutions quickly adjust to a 50\% increase in excitatory conductances. Simulations were performed using a custom Euler-Maruyama code on Matlab, with independent, Gaussian initial conditions $\mathcal N(3,1)$ for the voltage and $\mathcal N(2,1)$ for the adaptation variable. Conductances uniformly distributed in $[0,2]$ (excitatory) or $[0,3]$ (inhibitory). Other parameters: $\gamma(N)=N/10$, with $f(x)=x(1-x)(x-0.3)$,  $a=0.4$, $b=1.5$, $c=1$ and $\sigma=1$. }}
\label{fig:FhNChemical}
\end{figure}

\section{A few mathematical viewpoints on the problem}\label{sec:viewpoints}
\review{There are many ways to approach the conjecture formulated in this paper mathematically. Here, we discuss a possible approach that relies on reformulating the conjecture as a double-limit problem. In detail, the system considered here is a particular asymptotic regime of the equation:
\begin{equation}\label{eq:DoubleLim}
dx^i_t = \left(f_{p_i}(x^i_t)+\gamma\sum_{p=1}^{P} \sum_{j=1}^n \frac{{g}_{p_i,p}}{n} \,b_{p_i,p}(x^i,x^j) \right)\,dt+ \sigma_{p_i} dW^i_t. 
\end{equation}
where $n$ and $\gamma$ simultaneously go to infinity. Decoupling the divergence of $\gamma$ and $n$ in that fashion is not a mere artificial reformulation of the system: it opens the way to treat the conjecture as a double-limit problem (see diagram in Figure~\ref{fig:doubleLimit}). It heuristically sheds light on the phenomena that combine into the conjectured convergence to a balanced regime. In this diagram, there are at least two ways to approach the limit at hand: a first method (Method M1) consists in taking first the limit $n\to\infty$ with $\gamma$ fixed (going left to right in the diagram of Figure~\ref{fig:doubleLimit}, or step M1.a), and then taking the limit $\gamma\to \infty$ (M1.b). The second consists of taking the limit $\gamma\to\infty$ first, and then $n\to\infty$ (Method M2). We discuss what each of these limits will tell us about the system, as well as possible promising mathematical approaches to handle these questions.}
\smallskip

\begin{figure}[h]
\centerline{\includegraphics[width=.6\textwidth]{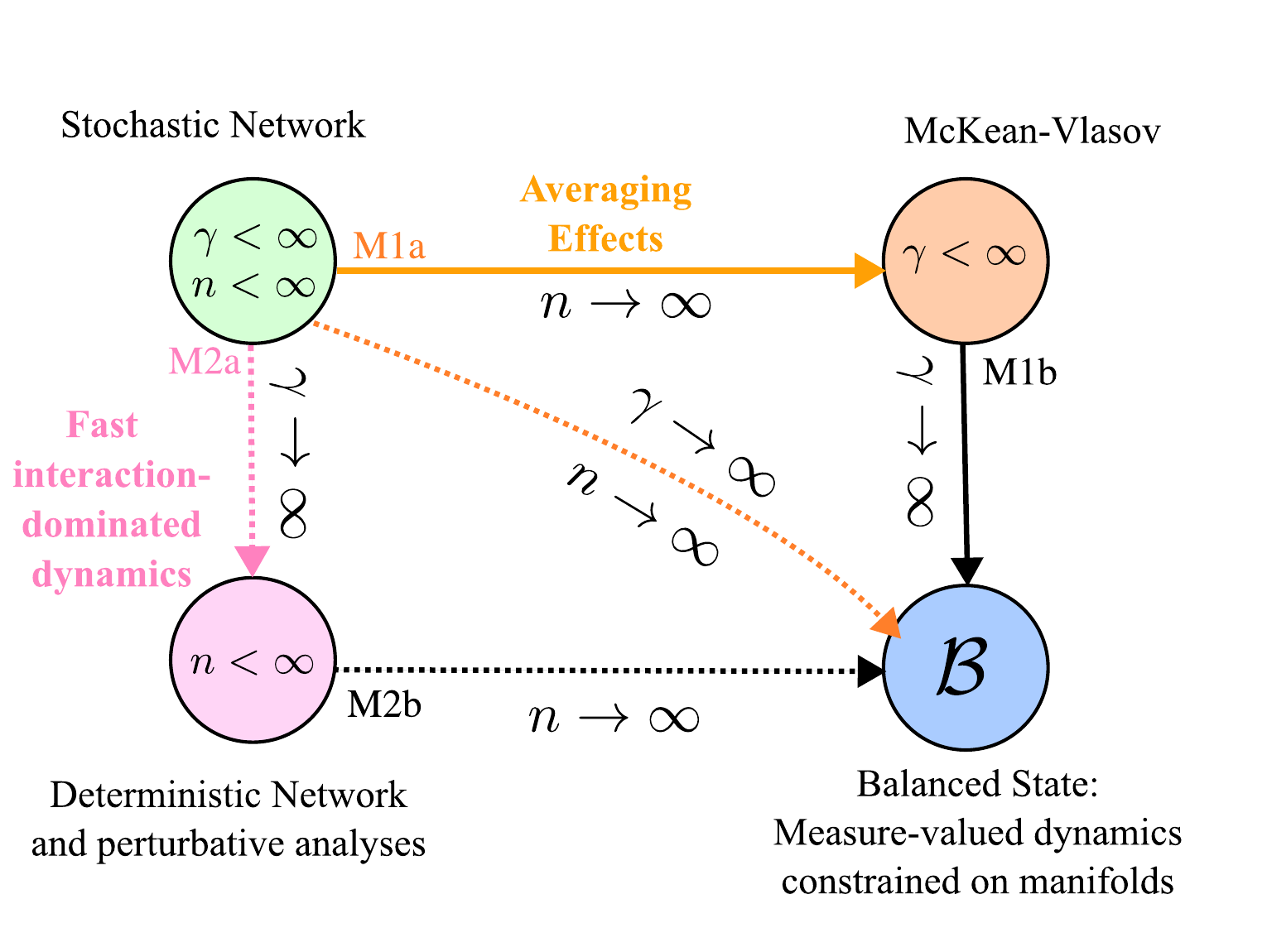}}
\caption{Double-limit perspective on the balance conjecture. The original setup of the problem can be viewed as a specific path within the double-limit problem, where $n \ to \ infty$ and $\ gamma \ to \ infty$ separately, as shown in equation~\eqref{eq:DoubleLim}. The limit $n\to \infty$ with $\gamma$ fixed (step M1a) typically yields, under appropriate assumptions, McKean-Vlasov mean-field equations that reflect averaging effects (orange circle). The limit $\gamma\to\infty$ with $n$ fixed rather corresponds to a singular value problem, whereby a change of time shall provide us with a system of ODEs informing us about early behaviors (arising, when $\gamma$ is large but finite, at a timescale of $\gamma^{-1}$, pink circle and step M2a). The problem discussed in conjecture~\ref{conjecture} is represented by the blue circle, and can be seen as a slanted path (orange arrow) in this double-limit graph, which we conjecture will collapse onto the balanced manifold $\mathcal{B}$. The rigorous results demonstrated in the paper follow the solid arrows (and hold under additional assumptions). }
\label{fig:doubleLimit}
\end{figure}

\begin{enumerate}[(i)]\itemsep=2pt
\item[M1.a] \review{Taking the limit $n\to \infty$ first puts the emphasis on averaging effects first, and then explores how collective dynamics adjust to diverging coupling. The first step in this approach, which involves taking the $n\to\infty$ limit in a set of interacting stochastic agents, is a classical question in stochastic analysis that has been widely studied over the past few decades and is relatively well understood. Under mild conditions on the system's parameters, the empirical measure of the collection of neurons in a given population, as well as the probability distribution of the state of each agent, converges to McKean-Vlasov processes, described by stochastic differential equations whose coefficients depend on the distribution of the solution itself. This thus yields a family of McKean-Vlasov processes indexed by $\gamma$. 
\item [M1.b] The second step in Method M1 consists of characterizing the adherent points and convergence properties of the family of processes as $\gamma$ diverges. Here again, a few methods are likely to apply, either working in the space of measures endowed with appropriate distances (typically, the Wasserstein space~\cite{blaustein2023concentration,blaustein2023large,blaustein2024concentration}). An alternative method involves writing the sequence of probability densities of the system, typically described by McKean-Vlasov Fokker-Planck equations, which are partial integro-differential equations indexed by $\gamma$, and studying adherent points and convergence of the sequence of functions as $\gamma\to\infty$ in an appropriate functional space. }
\item[M2.a]  \review{The second view on the problem consists in taking the limit $\gamma\to\infty$ first, emphasizing balance phenomena before considering how large-scale interactions maintain or alter these dynamics. Here, we expect that a change of time, proportional to $\gamma^{-1}$ will desingularize the network equation. It will make each term of the interactions of order $1/n$ (and the whole interaction of order 1) instead of being proportional to $\gamma$. It will make the intrinsic dynamic term and the stochastic forcing small. On this system, we expect the theory of regular perturbations of stochastic differential equations to yield convergence to a deterministic network equation, given by an ordinary differential equation, which emphasizes the emergence of balance. This approach is illustrated in section~\ref{sec:HeuristicTime} in the case of the two-population FitzHugh-Nagumo network with chemical synapses. 
\item[M2.b] Asymptotic regimes of these network equations as $n\to \infty$ and convergence properties {now} need to be identified. Asymptotic limits of deterministic systems are also well developed and should be available for the system.}
\end{enumerate}
\review{ Of note, one difficulty in method M2 consists in understanding what happens at intermediate times. Indeed, after taking the limit $\gamma\to\infty$, the description obtained after step M2a is valid on an infinitesimal time interval, and the differential equations obtained no longer depend on variables or dynamical terms that were not part of the diverging coupling. As soon as the balance has been ensured in these systems, the subsequent dynamics are no longer described by the same equations and will need a specific treatment. In that view, this double-limit is not a classical problem of commutation of limits, but rather two complementary views on the stochastic equation, where the limit $\gamma\to \infty$ informs us of the rapid collapses onto a balanced manifold, when it happens, yielding the system to be confined in the vicinity of the balanced regime where one can then explore the dynamics and averaging effects under the balance constraints.}

\review{The two subsections below discuss, in practice and formally, what both of the techniques outlined above entail. Due to the difficulties outlined regarding the dynamics at intermediate times in method M2, we focused our attention on method M1 in this paper. The asymptotic results we obtain for the sequence of McKean-Vlasov processes rely on functional analysis, more precisely on a desingularization technique that allows us to define viscosity solutions for the sequence of McKean-Vlasov processes, to prove that the sequence of probability densities associated are relatively compact, thus converging along subsequences, and to characterize some properties of the adherent points of the sequence (see section~\ref{sec:theory} and in the Appendices).}

\subsection{Method M1: Viscosity solutions of a family of McKean-Vlasov processes}\label{sec:methodHeuristic}
\review{Method M1 in the double-limit problem consists in first taking $n\to \infty$, before $\gamma\to \infty$. In that case, under appropriate assumptions on the functional parameters, equation~\eqref{eq:general_network} becomes a system of McKean-Vlasov equation~\eqref{eq:nonlinear_process}:
\begin{equation}
d\bar{x}_t^{p} = \left(f_p(\bar x_t^p)+\gamma \sum_{q=1}^P g_{p,q}\,\int_{y} b_{p,q}(\bar x_t^p,y)\mu_t^q(dy) \right)\,dt+ \sigma_p dW^p_t.
\end{equation}
This equation defines a solution dependent on the coupling parameter $\gamma$. The problem now amounts to studying this collection of processes and its asymptotic regimes when $\gamma\to\infty$. Considering limits of sequences of McKean-Vlasov processes is quite uncharted territory, and behaviors arising in the limits of those processes may be complex and associated with the emergence of singular distributions and Dirac masses. A technique we have been developing in~\cite{quininao2020clamping} to desingularize the behaviors has consisted of considering the McKean-Vlasov Fokker-Planck equation on the density of the process. For simplicity, we outline the methodology in the case of the FitzHugh-Nagumo equations with chemical synapses introduced above in equation~\eqref{eq:EIFhN}. This method will be rigorously applied to a one-dimensional system in Section~\ref{sec:theory}. }

\review{Under appropriate assumptions, in the limit $n\to \infty$ with fixed $\gamma$, neurons in population $p$ shall converge to independent copies of the McKean-Vlasov process given by the system of implicit stochastic equations:
\begin{equation}\label{eq:EIFhNMcKeanVlasov}
		\begin{cases}
			d\bar{x}^p_t&=\displaystyle{\Big(f(\bar{x}^p_t)-\bar{y}^p_t+\gamma g_{Ep} (\bar{x}^p_t-E_{E}) \mathbb{E}[\bar{s}^E_t]+\gamma g_{Ip} (\bar{x}^p_t-E_{I}) \mathbb{E}[\bar{s}^I_t]\Big)\,dt +\sigma dW^p_t}\\
			d\bar{y}^p_t&=a\Big(b\bar{x}^p_t-\bar{y}^p_t+c\Big)\,dt\\
			d\bar{s}^{p}_t&=\Big(-\frac{\bar{s}^{p}_t}{\tau}+\alpha(\bar{x}^p_t)(1-\bar{s}^{p}_t)\Big)\,dt.
		\end{cases}
	\end{equation}
The probability distributions $p_E(dx,dy,ds)$ and $p_I(dx,dy,ds)$ of neurons in the E/I populations, again under appropriate assumptions, satisfy a system of two coupled McKean-Vlasov-Fokker-Planck equations:}

	\begin{multline*}
		\partial_t p_\beta = -\Big(A_x(x,y)+\gamma g_{E\beta} \mathbb{E}[s_E]+\gamma g_{I\beta}\mathbb{E}[ s_I]\Big)p_\beta(x,y,s)\\
		-\Big(A(x,y)+\gamma g_{E\beta} (x-E_E)\mathbb{E}[ s_E]+\gamma g_{I\beta} (x-E_I)\mathbb{E}[ s_I]\Big)\partial_xp_\beta(x,y,s)+\frac{\sigma^2}{2}\partial_{xx}^2 p_\beta\\
		-B_y(x,y)p_\beta-B(x,y)\partial_y p_\beta-\Big(-\frac{1}{\tau}-\alpha(x)\Big)p_\beta-\Big(-\frac{s}{\tau}+\alpha(x)(1-s)\Big)\partial_sp_\beta.
	\end{multline*}
	for $\beta=\{E,I\}$, $A(x,y)=f(x)-y$, $B(x,y)=a(bx-y+c)$, and $\mathbb{E}[ s^\beta]:=\int_{\R^3} s p_\beta(x,y,s)$ is the probabilistic expectation of the synaptic variables. \review{We are interested in the limit $\gamma \to \infty$, and the application of the conjecture suggests concentration on singular, Dirac measures. To desingularize the problem, we shall use a Hopf-Cole transform (analogue to a Wentzel–Kramers–Brillouin (WKB) method well known to physicists)} $
	p_\beta=\exp\left(\gamma \varphi_\beta\right) $. We get:
	$$
	\frac{\partial_v p_\beta}{p_\beta} = \gamma \partial_v\varphi_\beta,\qquad \frac{\partial^2_v p_\beta}{p_\beta} = \left|\gamma \partial_v \varphi_\beta \right|^2+\gamma \partial^2_v \varphi_\beta,
	$$
 	we find that $\varphi_\beta$ is a solution to:
	\begin{multline*}
		\gamma \partial_t \varphi_\beta = -\Big(A_x(x,y)+\gamma g_{E\beta} \mathbb{E}[ s_E]+\gamma g_{I\beta} \mathbb{E}[  s_I]\Big)\\
		-\gamma \Big(A(x,y)+\gamma  g_{E\beta} (x-E_E) \mathbb{E}[  s_E]+\gamma g_{I\beta} (x-E_I)\mathbb{E}[  s_I]\Big)\partial_x\varphi_\beta
		\\
		+\frac{\sigma^2}{2} \left|\gamma \partial_x \varphi_\beta\right|^2+\frac{\sigma^2}{2}\gamma \partial^2_{xx} \varphi_\beta-B_y(x,y)-\gamma B(x,y)\partial_y\varphi_\beta\\
		+\frac{1}{\tau}+\alpha(x)-\gamma \Big(-\frac{s}{\tau}+\alpha(x)(1-s)\Big)\partial_s\varphi_\beta.
	\end{multline*}

Identifying the terms of the same order of magnitude in this expansion, we find, for the terms of order $\gamma^{2}$, the equation:
\[
\frac{\sigma^2}{2}(\partial_x \varphi_\beta)^2=\Big(g_{E\beta} (x-E_E)\mathbb{E}[  s_E]+g_{I\beta} (x-E_I)\mathbb{E}[  s_I]\Big)\partial_x\varphi_\beta.
\]

\review{This equation has two possible solutions: either $\partial_x \varphi_\beta=0$, so $\varphi=\kappa(y,s)$ a function independent of $x$. Such a solution would not be realistic because of the normalization condition for $p_\beta$. Consider now a point where $\partial_x \varphi_\beta\neq 0$. Assuming regularity, we have a non-trivial interval around that point where $\partial_x \varphi_\beta\neq 0$. Dividing by $\partial_x \varphi_\beta$ and integrating the differential equation then yields:}
\[
\varphi_\beta =\frac{1}{\sigma^2}\Big[g_{E\beta} (x-E_E)^2 \mathbb{E}[  s_E]+g_{I\beta} (x-E_I)^2\mathbb{E}[  s_I]\Big]+\kappa(y,s)
\]
In order to have solutions for large $\gamma$, the function $\varphi_\beta$ must be concave (as a function of $x$), requiring:
\[
\partial^2_x\varphi_E<0\quad\Rightarrow\quad g_{E\beta}  \mathbb{E}[s_E]+g_{I\beta}  \mathbb{E}[s_I]<0,
\]
which is precisely the condition for stability of the balance regime derived above and predicted by the conjecture. Classical theorems of concentration of measures would typically require $\varphi_E$ to be non-positive and ensure that, when $\gamma \to \infty$, the system converges to the zeros of $\varphi_E$. Because $\varphi_E$ is quadratic, this implies that the maximal value of $\varphi_E$ is equal to $0$. The maximal value is reached where
\[
\frac{2}{\sigma^2}\Big[g_{E\beta} (x-E_E)  \mathbb{E}[s_E]+g_{I\beta} (x-E_I) \mathbb{E}[s_I]\Big]=0, \]
that is, at a voltage $x^\star_\beta$ equal to:
\[x^\star_\beta= \frac{g_{E\beta} E_E  \mathbb{E}[s_E]+g_{I\beta}E_I  \mathbb{E}[s_I]}{g_{E\beta}  \mathbb{E}[s_E]+g_{I\beta}  \mathbb{E}[s_I]}
\]
is precisely the voltage corresponding to balance, as we derived above.

%\begin{theorem}
%When $\bar J=\eps^{-1}$, if there exists constants $C,C_\infty,c_\infty>0$ such that
%\begin{equation}\label{eq:fb_bounded}
%\|f\|_{H^{3,\infty}(\R)} \leq C,\quad \|b(x,y)\|_{H^{3,\infty}(\R^2)} \leq C,\quad c_\infty< b(x,y)< C_\infty
%\end{equation}
%for $i=\{1,2,3\}$ and if $p_\eps^0$ the set of initial conditions lies on $L^2_{m_\kappa}(\R)$, then for each $\eps>0$ there exists a unique solution $p_\eps\in C(\R^+;L^2_{m_\kappa}(\R))$ to equation~\eqref{eq:pde}, that satisfies. This solution is non-negative for all $t > 0$.
%\end{theorem}

\subsection{Method M2. Early behaviors and hydrodynamic limits}\label{sec:HeuristicTime}
\review{Method M2 by contrast, highlights the rapid convergence to balanced regimes, at timescales of order $1/\gamma$. Defining $\tilde{x}^i_t=x^i_{t/\gamma}$ where $x^i$ is a solution of equation~\eqref{eq:DoubleLim}, standard calculus yields: 
\begin{equation*}
d\tilde{x}^i_t =\left(\sum_{p=1}^P \sum_{j=1}^n g_{p_i,p} \,b_{p_i,p}(x^i,x^j) \right)\,dt + \frac{1}{\gamma} f_{p_i}(x^i_t)\,dt+ \frac{\sigma_{p_i}}{\sqrt{\gamma}} dW^i_t. 
\end{equation*}
This new equation makes it apparent that, at that timescale, the system may be seen as a small perturbation of a differential equation, closely resembling equation~\eqref{eq:earlyBehaviors}, which appears to amount to replacing the sum over all neurons by the statistical average over the whole population (which, if initial conditions are assumed independent, seems plausible). Here again for specificity, we formally apply these ideas to the FitzHugh-Nagumo system with chemical synapses given by equations~\eqref{eq:EIFhN}.} Defining $(\tilde{x}^i_t,\tilde{y}^i_t,\tilde{s}^i_t)=(x^i_{t/\gamma},y^i_{t/\gamma},s^i_{t/\gamma})$, we have:
\begin{equation}
		\begin{cases}
			d\tilde x^i_t&=\displaystyle{\Big(\frac{1}{n}\sum_{p=1}^P\sum_{j=1}^{n}g_{pp_i} (\tilde x^i_t-E_{p}) \tilde s^{j}_t+\frac 1 {\gamma} (f(\tilde x^i_t)-\tilde y^i_t)\Big)\,dt +\frac{\sigma}{\sqrt{\gamma}} d\tilde{W}^i_t}\\
			d\tilde y^i_t&=\frac 1 {\gamma} a\Big(b\tilde x^i_t-\tilde y^i_t+c\Big)\,dt\\
			d\tilde s^{i}_t&=\frac 1 {\gamma} \Big(-\frac{\tilde s^{i}_t}{\tau}+\alpha(\tilde x^i_t)(1-\tilde s^{i}_t)\Big)\,dt
		\end{cases}
	\end{equation}
	where $\tilde{W}^i$ are independent Brownian motions. Therefore, when $n\to\infty$, perturbation theory suggests that, at this timescale, $\tilde{x}$ and $\tilde{s}$ are constant while $\tilde x$ shall satisfy the equation:
 		\begin{equation}
		\nonumber \frac{d{\tilde{x}}^i_t}{dt}=g_{Ep_i} (\tilde x^i_t-E_{E}) \bar{s}^{E}_t +g_{Ip_i} (\tilde x^i_t-E_{I}) \bar{s}^{I}_t \\
		\end{equation}
	where we noted $p_i\in\{E,I\}$ the population of neuron $i$ and $\bar{s}^{p}$ the empirical average of the synaptic conductances for all neurons in population $p$ (here, simply determined by the distribution of initial conditions). This equation is identical to the one derived above following the conjecture, and thus the condition for convergence to the balanced regime ($(g_{E\beta} \bar{s}^E+g_{I\beta} \bar{s}^I)<0$) as well as the convergence of voltages to $(x^\star_E,x^\star_I)$ ensue. At this voltage, all neurons, both excitatory and inhibitory, receive a net input equal to $0$, corresponding to a balance condition. 
	
	In addition to providing heuristic support for the derivations above, these formal derivations also suggest that the timescale of collapse onto a balanced regime is proportional to $\gamma (g_{E\beta} \bar{s}^E+g_{I\beta} \bar{s}^I)$, it is thus arbitrarily fast as $\gamma$ increases when the convergence condition is satisfied. Moreover, a perturbation arising at any point in time and pulling the system away from the balanced regime is quickly absorbed at the same timescale.

\section{Theoretical Results}\label{sec:theory}
\review{In the previous section, we explored the consequences and different mathematical or theoretical approaches to approach the problem of convergence to balance in networks with diverging interactions. Mathematically, proving the result stated in conjecture~\ref{conjecture} seems out of reach of current methods. We present a proof of convergence to the balance regime in a special limit and for a class of one-dimensional models with separable coupling. The steps of the proof are outlined below, and mathematical details are provided in the Appendix sections A-C. }
%\subsection{Summary of the mathematical results}

\review{The asymptotic regime considered theoretically corresponds to the dynamics arising when taking the limit $n\to\infty$ first, and then $\gamma \to \infty$, in the description of figure~\ref{fig:doubleLimit}. We explored these asymptotic regimes in~\cite{quininao2020clamping} in the context of the model introduced in section~\ref{sec:ElectricalExample}. The limit $n\to\infty$ relied on previous results on the mean-field limit of the system proven earlier~\cite{mischler2015kinetic}, and the limit $\gamma\to\infty$ was proven using Hopf-Cole methods under the assumption that initial conditions show a Gaussian concentration to a Dirac mass as $\gamma\to\infty$. This blow-up method allowed showing that the sequence of measures collapsed to a Dirac mass, and that the blow-up profile had a Gaussian profile. }
\review{Here, we extend these results to a general, one-dimensional model with non-diffusive coupling. In this paper, we demonstrate that Hamilton-Jacobi approaches introduced in~\cite{quininao2020clamping} extend to a general model inspired by the chemical coupling setting introduced in section~\ref{sec:Chemical}. Similar to the systems studied in that section, we consider that interactions are non-diffusive but can be written as a separable function 
\[b(x_i,x_j)=\alpha(x_i)\beta(x_j)\] 
where $\alpha$ and $\beta$ are smooth functions and $\beta$ is bounded. \review{These hypotheses are inspired by chemical synapses, described in more detail in section~\ref{sec:Chemical}. In that model, the interaction term is the product of a linear function of the voltage $x^i$ with a dynamical synaptic variable $s^j$. In a quasi-static approximation of the synaptic variable,  $b(x_i,x_j)=(x_i-E)\,\frac{a(x_j)}{1/\tau+a(x_j)}$, where $a$ is the sigmoidal function that was noted $\alpha$ in equation~\eqref{eq:EIFhN}. }}

For simplicity, in this first study of this general problem, we consider a one-dimensional, one-population model ($d=1$, $P=1$, $\eps=\gamma^{-1}$, for consistency with the literature on viscosity solutions):
\begin{equation}\label{eq:pdeSummary}
\partial_t\mu_t = -\partial_x\left[\left(f(x)- \eps^{-1}\int_{y} b(x,y)\mu_t \right)\mu_t\right]+\frac{\sigma^2}{2}\partial^2_{xx}\mu_t, 
\end{equation}
%under the assumptions that:
%\begin{itemize}
%\item \emph{Intrinsic dynamics:} We assume that there is some constant $C_f>0$ such that:
%\begin{equation}
%\label{hyp:h1}
%yf(y) \leq C_f-y^2.
%\end{equation}
%This assumption allows using linearly extended third-degree polynomials $f$ as in~\cite{quininao2020clamping} but also works for any odd polynomial with a negative leading term inside a large compact, and extended linearly to the whole real line. 
%\item the function $\alpha$ modeling the dependence in post-synaptic neuron voltage that is at most linearly increasing at infinity, that is, there exists $C_\ell>0$ and $C_r >0$ such that
%\begin{equation}
%\label{hyp:h2}
%\lim_{y\rightarrow-\infty}\alpha(y)y^{-1}=C_\ell,\quad \lim_{y\rightarrow+\infty}\alpha(y)y^{-1} = C_r.
%\end{equation}
%In the case of chemical synapses, $\alpha(x)$ is typically a linear function where these assumptions hold. 
%\item the function $\beta$ modeling the dependence of the interaction function in post-synaptic neuron voltage is positive and at most polynomial at infinity. In detail, there exists an integer $k\in\mathbb N$ and two constants $K,C_k>0$ such that
%\begin{equation}
%\label{hyp:h3}
%0<K^{-1}\leq \beta(y)\leq C_k(1+y^{k}).
%\end{equation}
%This, again, is typically satisfied for synapse models, where interactions are positive and upper-bounded.
%\end{itemize}

\review{We prove here the existence and uniqueness of weak solutions to the nonlinear equation~\eqref{eq:pde} in an appropriate functional space. Similar to the functional setting proposed in~\cite{mischler2015kinetic}, we work in the space of functions $L^2_{m_\kappa}(\R)$ with $m_\kappa = e^{\kappa(1+x^2)}$ for some  $\kappa\in(0,1)$, which is composed of all functions $u$ with finite weighted norm
 $$
 \|u\|_{L^2_{m_\kappa}}:= \left(\int u(y)^2\,m_\kappa(y) dy\right)^{1/2}.
 $$
 \review{Working in this space allows ensuring that any solution will have sub-Gaussian or Gaussian tails, which is particularly relevant in the case of the FitzHugh-Nagumo system, where a quartic potential confines Brownian particles, so it is expected that tails of the distribution will not be heavy. Making sure that tails decay fast enough further allows considering moments of any degree and expected values of exponential functions, as long as their divergence at infinity is upperbounded by $m_\kappa$. }
 \begin{theorem}
\label{th:MR-existence}
Assume that $b(x,y)=\alpha(x)\beta(y)$ with $\beta$ strictly positive and upper-bounded by a polynomial function. Assume furthermore that $f$ and $\alpha$ are $C^1$ functions such that there are some positive constants $C_0,C_1$ and $C_2$ such that
\begin{equation}
\label{hyp:mr-existence}
f'(x)\leq C_0(1-x^2),\quad \lim_{x\rightarrow-\infty}\alpha'(x)=C_1, \quad \lim_{x\rightarrow+\infty}\alpha'(x)=C_2,
\end{equation}
and that, for all $\eps>0$, the initial conditions $\mu_{0,\eps}$ belong to the ball of radius $K_0$ of $L^2_{m_1}(\R)$. Then for each $\kappa\in(0,1)$ and $\eps>0$ there exists a unique solution $\mu_\eps\in C(\R^+;L^2_{m_\kappa}(\R))$ to equation~\eqref{eq:pde}. This solution satisfies the mass conservation principle and is nonnegative for all $t>0$.
\end{theorem}
\review{The proof of this result is deferred to Appendix~\ref{sec:priori}. It proceeds as follows:
\begin{itemize}
\item We start by showing that, under assumption~\eqref{hyp:mr-existence}, the interaction term $I_{\mu_\eps}(t)=\int \beta(y)\mu_{\eps}(t,dx)$ is uniformly bounded. 
\item This allows us to initiate a Picard fixed-point method to show well-posedness of the problem. The boundedness of the interaction term ensures that a map, of which solutions to the McKean-Vlasov problem are fixed points, is defined in the appropriate space, allowing us to iterate this map to create Picard sequences.
\item We next show that the map, now linear, is a contraction. 
\end{itemize} 
As a consequence, we obtain the existence of a unique weak solution to equation~\eqref{eq:pde}.
We next turn to studying the sequence of solutions as $\gamma=\epsilon^ {-1}$ varies. Lemma~\ref{lem:Jeps} shows that for a particular class of functions $\alpha$ and $\beta$, and after extracting a subsequence, $I_\eps(t)$ converges almost everywhere in $t$ to a function $I(t)$, as $\eps\rightarrow0$. This result is a key step in the subsequent analysis of the Hopf-Cole transformation of the solutions to the system, defined by $\phi_\eps(t,x)=\eps\log \mu_{\eps}(t,x)$,  which is the focus of the second main result:}
\begin{theorem}
\label{th:MR-convergence}
Under the same hypotheses of Theorem~\ref{th:MR-existence}, let $\mu_\eps(t,x)$ the solution to~\eqref{eq:pdeSummary} with initial conditions $\mu_{\eps,0}(x)$ satisfying the condition
$$
\sup_{0<\eps<1}\eps\log\mu_{0,\eps}(x)\leq -Ax^2+B,
$$
for some constants $A,B$ with $A>0$. Then the family of functions $\phi_\eps(t,x)$ are well defined and we have the following
\begin{itemize}
\item Both sequences $(\phi_\eps)_{\eps\in(0,1)}$ and $(I_{\mu_\eps}(t))_{\eps\in(0,1)}$ are relatively compact.
\item Let $I_{\mu_{\eps_n}}(t)$ be a sequence converging to a function $I(t)$. There is a subsequence $\{\eta_n\}\subset\{\eps_n\}$ and a continuous function $\phi(t,x)$ such that $\phi_{\eta_n}\rightarrow\phi$ and $\phi$ is a viscosity solution to 
$$
-\alpha(x)I(t)\partial_x\phi(t,x)-\frac{\sigma^2}{2}|\partial_x\phi(t,x)|^2= 0.
$$
\item Along a convergent subsequence $\phi_{\eps_n}$, the support of $\mu_{\eps_n}$ shrinks, \review{a.e. in $t$,} on the zero set of $\phi$.
\end{itemize}
\end{theorem}
\review{To prove this result, we then study in Appendix~\ref{sec:hopf} some regularity estimates obtained from maximum/comparison principles for elliptic equations. These technical steps will provide us with the main elements used to prove, in Appendix section~\ref{sec:limit}, Theorem~\ref{th:MR-convergence}. More precisely, noting that $\mu_\eps=\exp(\phi_\eps/\eps)$ for the limit to exist as $\eps\rightarrow0$, it is necessary that $\phi_\eps$ becomes a nonnegative function. The steps of the proof begin with Proposition~\ref{prop:supersol}, where we show that $\phi_\varepsilon$ is uniformly bounded from above by a function resembling a quadratic polynomial with a negative leading coefficient outside a large compact set. Next, in Proposition~\ref{prop:regularity_space}, we analyze the spatial derivatives of $\phi_\varepsilon$ and deduce that the family $\{\phi_\varepsilon\}$ is uniformly bounded on compact subsets of $(0, \infty) \times \mathbb{R}$. We also investigate the time regularity of the family $\{\phi_\varepsilon\}$.
With all these auxiliary results, we establish that for any adherent point of the set $\{I_\varepsilon\}$, there exists a subsequence $\phi_{\eps}$ converging to a continuous function $\phi(t,x)$, and we characterize those limits.}
\color{black}}

\section{Discussion and Conclusion}
\review{The mathematical theory of large-scale systems of interacting agents has been largely focused on regimes where interactions between neurons become infinitesimal in the large $n$ limit, favoring averaging effects~\cite{arous1995large,arous1997symmetric,baladron2011mean,cabana1,cabana2,robert2014dynamics}, and preventing any possible divergence of the interaction term. Recently, Ramanan and collaborators dropped the hypothesis of infinitesimal interactions, and investigated the behavior of stochastic particle systems on sparse graphs such that each neuron is expected to receive a finite input even in the asymptotic regime~\cite{ramanan2023interacting,ramanan2023large,ganguly2024hydrodynamic,lacker2023local,ramanan2022beyond}. These novel techniques rely on local limits of random graphs, Markov random fields, and stochastic analysis. While interactions are kept non-vanishing, the expected interaction term remains finite again because of the dilution of the graph, with a connection probability between any two agents decaying as the network size increases.}

\review{ The regime we explored in the present manuscript is different from these works in that we consider that neither the scaling of the synaptic weights nor the scaling of the connection probability ensures boundedness of the interaction term. In our case, under appropriate assumptions, the entire dynamics must conspire to maintain finite interactions for the equation not to diverge. In this regime, conjectured by Van Vreeswijk and Sompolinsky in~\cite{vreeswijk1998chaotic}, those scalings are indeed too slow (in terms of decay according to the network size) to ensure boundedness of the interaction term. Phenomenologically, Sompolinsky and Van Vreeswijk instead proposed that in this case, the dynamics will first quickly evolve to cancel these interactions (that is, \emph{balance} excitation and inhibition), when possible and when the dynamics allow, leading the system to collapse rapidly on a statistical manifold where the interactions are kept in check, and where non trivial (in the original paper, chaotic) dynamics can emerge. It is this regime that we studied in the present paper. From a mathematical viewpoint, the classical \emph{mean-field} techniques reviewed above, which apply to systems with vanishing interaction terms or vanishing connection probabilities and reveal averaging effects, are inoperant in the present setting. This question raises several new mathematical problems that are not addressed by any existing theory, to the best of our knowledge. This is why we initially presented the problem in its full generality and formulated a relatively general mathematical conjecture. Proving this conjecture, or some aspects of it, constitutes an interesting challenge that shall combine stochastic analysis techniques with dynamical systems techniques and analysis. Although we had not yet identified the generality of the problem, we first analyzed problems related to these regimes of activity in~\cite{quininao2020clamping} in a particular neural network. In that paper, we proved a concentration of the distribution of mean-field solutions when coupling diverged, and to that purpose used asymptotic desingularization techniques for PDEs that had been introduced in the context of mathematical ecology~\cite{barles2009concentration,barles2002geometrical,mirrahimi2015uniqueness,bouin2013hamilton,lorz2011dirac}. These results were then proven using other stochastic analysis techniques by Blaustein and Filbet, including the control of the convergence in the Wasserstein space (allowing them to lift some assumption on initial conditions~\cite{blaustein2023concentration}) or achieve a stronger notion of convergence~\cite{blaustein2023large}, where interactions were extended to include a dependence in space. All of these works relied substantially on the straightforward nature of the interaction term and are not trivial to generalize to other types of couplings. }

\review{In terms of proofs, while we were able in the present manuscript to generalize desingularization techniques to a one-dimensional system with separable coupling, we believe the desingularization approach based on Hopf-Cole transforms (or similar) is bound to be limited to a narrow set of models. We proposed here two theoretical points of view on the conjecture, by formulating it as a double-limit problem in the space of probability measures on continuous functions (or, in the space of random processes). This yields two distinct approaches. The first limit consists of considering what happens at very short timescales (on the order of the typical divergence of the interaction term). As timescales vanish with the network size, one obtains a typically simple, non-stochastic dynamical system where all variables are fixed except those affected by the divergence of the interaction term. The balanced regime may appear, from this lens, as a dynamical equilibrium of the system at short timescales. The other limit considers first that the network size diverges, and then takes the interaction coefficient to infinity. This provides two avenues to approach the conjecture. This lifts part of the veil on a new problem and a new phenomenology that are well worth studying mathematically and computationally. }

\review{From the application's point of view, in neuroscience, the importance of this phenomenon and of its consequences is hard to overstate. After its first formulation by Sompolinsky and Van Vreeswijk, it gave rise to a vast literature that explored balanced networks in rhythm generation~\cite{brunel2000dynamics,brunel2000dynamics2}, cluster formation with plasticity~\cite{litwin2014formation,litwin2012slow}, pattern formation~\cite{rosenbaum2014balanced}, and computational capacities of learning networks~\cite{leone2025noise,kadmon2020predictive,stock2022synaptic}. Biologically, this theory of a dynamical balance can account for many of the observations made in the maintenance of balance in networks subject to perturbations~\cite{dai2016maintenance,chen2022homeostatic}, but does not provide an answer as to which of the many possible mechanisms are at play in this maintenance~\cite{turrigiano47keeping,turrigiano2003homeostatic,turrigiano2011too,turrigiano2008self}, and how, at the end of the day, indivual neurons adjust to reach the balanced manifold.}

\review{While not explicitly stated as balanced regimes, several recent works on particle systems tend to reflect the emergence of a collapse to a balance manifold, similar to the result we formulate here. In particular, Bruna, Burger, and De Wit~\cite{bruna2025lane} recently proposed a mathematical model of a foraging colony of ants, where each ant is modeled as a Brownian particle and interactions are modeled through chemotaxis, wherein ants align their orientations upwards towards the gradient of pheromone they release. In addition to some more classical behaviors of typical Keller-Segel type equations~\cite{arumugam2021keller}, they revealed the emergence of singular distributions for larger interaction strengths, including the formation of lanes and Dirac distributions, which are evocative of the balanced manifold predicted in this manuscript. Similarly, polarization phenomena have recently been linked to the stronger type of interactions that may arise in the modern world~\cite{stella2019influence}, where stronger interactions are reported to lead to more segregated components and fewer bridges between them. In ecology, such balances have garnered more interest in recent years~\cite{qian2020balance,neutel2014interaction,mccann1998weak,brooker1998balance}, and theoretically seem to provide a foundation for addressing possible questions related to balancing cooperation and competition in ecosystems subject to severe external environments.  Because the nature of the limiting object and its dynamics are drastically different in balanced systems compared to mean-field regimes, this paper argues for a reevaluation of the choice of models and for exploring which model better describes the collective behavior of large systems in applications, as well as for theoretical statistical models. Among choice models in mathematics, the study of spin glasses has been an active area of research, providing a number of techniques and results to the broader community working on statistical models. Incidentally, just less than a week ago, at the time of this writing, Bates and Sohn introduced a balanced multi-species spin glass model~\cite{bates2025balanced}, the study of which may provide an ideal framework for analyzing these new dynamics. }

{\bf Acknowledgements:} The authors warmly thank the anonymous referee and the Editor, whose suggestions allowed us to improve the manuscript deeply. 

\newpage
\appendix

\section{A priori bounds in the separable case}
\label{sec:priori}

We now establish part of the theory in a simplified regime where $d=1$, $P=1$ and $J_{ij}=\frac{\gamma}{n}$. Here, the coupling coefficient $\gamma$ diverges, and, to align our work with the classical literature on viscosity solutions of functional equations, we define $\eps=\gamma^{-1}>0$, so the large coupling regime corresponds to the limit  $\eps\to 0$. 

Working in the limit $N\to \infty$, we focus our attention on the McKean-Vlasov process given by \eqref{eq:nonlinear_process}, which, in the current setting, simplifies to:
\begin{equation*}
d\bar{x}_t = \left(f(\bar x_t) + \frac 1 \eps \int_{y} b(\bar x_t,y)\mu_t(dy) \right)\,dt+ \sigma dW_t
\end{equation*}
Classically, applying It\^o's lemma to the stochastic process $\varphi(x_t)$ for $\varphi$ a regular test function, we obtain the weak formulation of the Fokker-Planck equation. In detail, It\^o's lemma yields:
$$
d\varphi(x_t) = \varphi'(x_t)\left(f(x_t)- \frac 1 \eps \int_{y}\,b(x_t,y)\mu_t(dy) \right)\,dt+\sigma\varphi'(x_t)dW_t+\frac{\sigma^2}{2}\varphi''(x_t)dt,
$$
and taking the expectation on both sides of this equality, we get:
$$
\frac{d}{dt}\,\Ee[\varphi(x_t)] = \Ee\left[\varphi'(x_t)\left(f(x_t)- \eps^{-1}\int_{y} b(x,y)\mu_t(dy) \right)+\frac{\sigma^2}{2}\varphi''(x_t)\right].
$$
Assuming now that the measure $\mu$ has a density with respect to Lebesgue's measure, and writing, with a slight abuse of notation, $\mu_t(dx)=\mu_t(x)\,dx$, we obtain:
\begin{multline*}
 \int_\R \varphi'(x)\left(f(x)- \eps^{-1}\int_{y} \,b(x,y)\mu_t(dy) \right)\mu_t(dx)+\frac{\sigma^2}{2}\int_\R\varphi''(x) \mu_t(dx)
 \\
 =-\int_\R \varphi(x)\partial_x\left[\left(f(x)- \eps^{-1}\int_{y} b(x,y)\mu_t(dy) \right)\mu_t(dx)\right]+\frac{\sigma^2}{2}\int_\R\varphi(x) \partial^2_{xx}\mu_t(x)\,dx,
\end{multline*}
leading to observe that $\mu_t$ is a (weak) solution of the equation:
\begin{equation}\label{eq:pde}
\partial_t\mu_t = -\partial_x\left[\left(f(x)- \eps^{-1}\int_{y} b(x,y)\mu_t \right)\mu_t\right]+\frac{\sigma^2}{2}\partial^2_{xx}\mu_t.
\end{equation}

In this section, as indicated in the main text, we focus our attention on cases where the interaction function $b(x,y)$ is separable, i.e.
$$
b(x,y) = \alpha(x)\beta(y).
$$
\review{This assumption was made in line with typical models of chemical synapses (as model~\eqref{eq:EIFhN}) and for its mathematical convenience. Indeed, in model~\eqref{eq:EIFhN}, the interaction is given by a separable function, the product of a linear function of the voltage of the post-synaptic neuron by the synapse variable of the pre-synaptic neuron.
% yielding if both the recovery variable and the synapse are just a function of voltage, we can rewrite equation by
%\[
%dx^i_t=\displaystyle{\Big(f(x^i_t)+\frac{\gamma(N)}{N}\sum_{j=1}^{N}G (x^i_t-E) \beta(x^{j}_t)\Big)\,dt +\sigma dW^i_t},
%\]
%at the limit $N\rightarrow\infty$ the interaction term takes the form $\int_y\alpha(x)\beta(y)\mu_t$.
}

Throughout the section, we will use the following consequences of the set of hypotheses on Theorem~\ref{th:MR-existence}:
\begin{itemize}
\item \emph{Intrinsic dynamics:} From the bound on the first derivative of $f$, there exists a constant $C_f>0$ such that for any positive $x$:
\begin{equation}
\label{hyp:h1}
xf(x) = x\int_0^x f'(y)dy +xf(0) \leq  C_0x\left(x-\frac{x^3}{3}\right) +xf(0) \leq C_f-x^2,
\end{equation}
and a similar computation for negative $x$. 
%
%This assumption allows using continuously differentiable functions $f$ piecewise defined as third-degree polynomials in a bounded interval and linear outside that interval, as we did in~\cite{quininao2020clamping} (they would also work for any function that are the restriction of an odd-degree polynomial with a negative leading term inside a large compact, and extended linearly to the whole real line outside of that compact). 
%
\item \textit{Interaction term:} the function $\alpha$ modeling the dependence in post-synaptic neuron voltage that is at most linearly increasing at infinity. This result is direct from the bounds on $\alpha'$ of~\eqref{hyp:mr-existence}. Indeed, from L'H\^opital's rule, we have that
$$
\lim_{x\rightarrow-\infty}\frac{\alpha(x)}{x} = \lim_{x\rightarrow-\infty}\alpha'(x) =C_1,
$$
therefore
\begin{equation}
\label{hyp:h2}
\lim_{x\rightarrow-\infty}\alpha(x)x^{-1}=C_1,\quad \lim_{x\rightarrow+\infty}\alpha(x)x^{-1} = C_2.
\end{equation}
\review{In the case of chemical synapses, $\alpha(x)$ is typically a smooth or even a linear function where these assumptions hold~\cite{destexhe1998kinetic,yao2023plasticity}. }
\item \textit{Postsynaptic influence:}  the function $\beta$ modeling the dependence of the interaction function in post-synaptic neuron voltage is positive and at most polynomial at infinity. In detail, there exists an integer $k\in\mathbb N$ and two constants $K,C_k>0$ such that
\begin{equation}
\label{hyp:h3}
0<K^{-1}\leq \beta(y)\leq C_k(1+y^{k}).
\end{equation}
This, again, is typically satisfied for synapse models, where interactions are positive and upper-bounded.
\end{itemize}

%Unless necessary, we will use $C$ as a generic positive constant.

%We start our proof by controling start by bounding moments of $\mu_\eps$, estimates that will be useful in establishing our results.  For simplicity of notations, we define $I[u(t,\cdot)]$ (when no ambiguity is possible, noted $I_u$) as 
%\[I[u(t,\cdot)]= \int \beta(y)u(t,y)\,dy.\]
%
%With this notation, equation~\eqref{eq:pde} becomes
%\begin{equation}\label{eq:pde}
%\partial_t\mu_\eps = -\partial_x\left[\left(f(x)- \eps^{-1}\alpha(x)I_{\mu_\eps}(t) \right)\mu_\eps|\right]+\frac{\sigma^2}{2}\partial^2_{xx}\mu_\eps.
%\end{equation}
%To that aim, we start with the following Lemma:
\begin{lemma}
\label{lem:y2}
Assume that $f$ satisfies~\eqref{hyp:h1}, and that $b(x,y)=\alpha(x)\beta(y)$ with $\alpha$ satisfying~\eqref{hyp:h2}, and $\beta$ satisfying~\eqref{hyp:h3} for some $k\in\mathbb N$. Moreover, assume that the initial condition $\mu_\eps^0$ are such that there exists $K_0>0$ such that, for all $\eps>0$:
%and consider that the initial conditions have $2k$ moments bounded by a constant $K_0>0$ independent of $\eps$, (i.e., 
\[\int y^{2k}\mu_\eps^0(dy)<K_0.\] 
Then $\mu_\eps$ has uniformly bounded moments of order $2k$ and $I_{\mu_\eps}$ is uniformly bounded.
%$$
%\left\{\varphi_\eps(t)\right\}_{\eps>0},\text{\quad with\quad}\varphi_\eps(t):=\int y^{2k}\mu_\eps(t,dy),
%$$
% and $\{I_\eps(t)\}_{\eps>0}$ are uniformly bounded in $\eps$.

%More precisely, there is some positive constant $K_2$ such that
%\begin{equation}\label{eq:y2bound}
%0\leq \varphi_\eps(t)\leq \max(K_0,K_2),\qquad 0\leq I_\eps(t)\leq C+\max(K_0,K_2),\qquad \forall\eps>0.
%\end{equation}
\end{lemma}

\begin{proof}
Using equation~\eqref{eq:pde}, we have:
\begin{align*}
\frac{d}{dt}\int y^{2k}\mu_\eps(t,dy) &= \int 2k y^{2k-1}\Big(f(y)-\eps^{-1}\alpha(y)I_{\mu_\eps}(t)\Big)\mu_\eps(t,dy)
\\
&\qquad+\sigma^2\int k(2k-1)y^{2k-2}\mu_\eps(t,dy)
\\
&=k\int \Big(2 y^{2k-1}f(y)+\sigma^2(2k-1)y^{2k-2}\Big)\mu_\eps(t,dy)
\\
&\qquad-\frac{2kI_{\mu_\eps}(t)}{\eps}\int y^{2k-1}\alpha(y)\mu_\eps(t,dy).
\end{align*}
Using assumption~\eqref{hyp:h1}, we get:
\begin{align*}
2 y^{2k-1}f(y)+\sigma^2(2k-1)y^{2k-2}\leq 2y^{2k-2}(C_f-y^2)+\sigma^2(2k-1)y^{2k-2}\leq C'-y^{2k},
\end{align*}
with $C'$ a positive constant. 

Moreover, assumption~\eqref{hyp:h2} ensures the existence of $C_+,M_+\gg1,$ such that
$$
C_+\leq\frac{\alpha(y)}{y}\quad\Rightarrow\quad\int_{M_+}^{+\infty} y^{2k-1}\alpha(y)\mu_\eps(t,dy) \geq C_+\int_{M_+}^{+\infty} y^{2k}\mu_\eps(t,dy).
$$
Similarly, there exists $C_->0,M_-\ll-1$ such that
$$
C_-\leq\frac{\alpha(y)}{y}\quad\Rightarrow\quad\int_{-\infty}^{M_-} y^{2k-1}\alpha(y)\mu_\eps(t,dy) \geq C_-\int_{-\infty}^{M_-} y^{2k}\mu_\eps(t,dy).
$$
Let now $C_1=\min(C_-,C_+)>0$. By continuity, the function $y\mapsto y^{2k-1}\alpha(y)-C_1 y^{2k}$ is bounded on the interval $[M_-,M_+]$. Let now
\[C_2:=\int_{-M}^M(y^{2k-1}\alpha(y)-C_1 y^{2k})\mu_\eps(t,dy)<\infty.\]
We have:
$$
\int y^{2k-1}\alpha(y)\mu_\eps(t,dy)\geq C_2+C_1 \int y^{2k}\mu_\eps(t,dy).
$$
Altogether, we obtained that
\begin{align*}
\frac{d}{dt}\varphi_\eps(t)
% &=k\int \Big(2 y^{2k-1}f(y)+\sigma^2(2k-1)y^{2k-2}\Big)\mu_\eps(t,dy)
%\\
%&\qquad-\frac{2kI_{\mu_\eps}(t)}{\eps}\int y^{2k-1}\alpha(y)\mu_\eps(t,dy)
%\\
% &\leq k \int \Big(C'-y^{2k}\Big)\mu_\eps(t,dy) - \frac{2kI_{\mu_\eps}(t)}{\eps}\int y^{2k-1}\alpha(y)\mu_\eps(t,dy)
%\\
&\leq k(C' - \varphi_\eps(t))-\frac{2kI_{\mu_\eps}(t)}{\eps}\left(C_2+C_1\varphi_\eps(t)\right).
\end{align*}

Denote $C=\max(C',-C_2/C_1)$. It is easy to see that whenever $\varphi_\eps(t)\geq C$, its time derivative is strictly negative, so $\varphi_\eps$ will be decreasing in time, readily implying
 $$
 \varphi_\eps(t)\leq \max\left(K_0,C\right),\qquad \forall \eps>0.
 $$
%
%To finish the proof, we need to find some positive value $\tilde C$ independent of $\eps$ and $I_{\mu_\eps}(t)$, such that the righthand side of the previous inquality becomes strictly negative for all values $\varphi_\eps(t)>\tilde C$. This would imply that as soon as $\varphi_\eps(t)$ crosses $\tilde C$, its time derivative becomes negative, and then $\varphi_\eps(t)$ must decrease going back to values smaller than $\tilde C$.
%
%To that aim, notice that for $\varphi_\eps(t)>C'$, we can neglect the contribution of $C'-\varphi_\eps(t)$, and only focus on the term involving $I_\eps(t)$. Notice that if $\varphi_\eps(t)$ is larger than $-C_2/C_1$, then we have that
%$$
%\frac{2kI_{\mu_\eps}(t)}{\eps}\left(C_2+C_1\varphi_\eps(t)\right) \geq 0,
%$$
%then, in that case, we can also neglect that term's contribution, and the time derivative of $\varphi_\eps(t)$ becomes negative. By taking $\tilde C=\max(C',-C_2/C_1)$ we deduce that
% $$
% \varphi_\eps(t)=\int y^{2k}\mu_\eps(t,dy)\leq \max\left(K_0,\tilde C\right),\qquad \forall \eps>0.
% $$
 Finally, thanks to~\eqref{hyp:h3}, the positivity of $\mu_\eps(t,dy)$ and the mass conservation principle, we have that for all $\eps>0$ and for all $t>0$:
 $$
0\leq I_{\mu_\eps}(t) = \int\beta(y)\mu_\eps(t,dy) \leq \int C_k\left(2+y^{2k}\right)\mu_\eps(t,dy)\leq C_k+\max\left(K_0,\tilde C\right)
 $$
 $\qquad$
 \end{proof}
% \textcolor{blue}{[Note for myself: we did not use $\beta>0$ here]}
 
We next prove Theorem~\ref{th:MR-existence}, stating existence and uniqueness of weak solutions to the nonlinear equation~\eqref{eq:pde} in $L^2_{m_\kappa}(\R)$.
%\begin{theorem}
%\label{th:existence}
%We work under the standing assumptions of this section. Assume moreover that $f$ and $\alpha$ are $C^1$ functions such that there are some positive constants $C_0,C_1$ and $C_2$ such that
%\begin{equation}
%\label{hyp:h3prim}
%f'(x)\leq C_0(1-x^2),\quad \lim_{x\rightarrow-\infty}\alpha'(x)=C_1, \quad \lim_{x\rightarrow+\infty}\alpha'(x)=C_2,
%\end{equation}
%and that, for all $\eps>0$, the initial conditions $\mu_{0,\eps}$ belong to the ball of radius $K_0$ of $L^2_{m_1}(\R)$. Then for each $\kappa\in(0,1)$ and $\eps>0$ there exists a unique solution $\mu_\eps\in C(\R^+;L^2_{m_\kappa}(\R))$ to equation~\eqref{eq:pde}. This solution satisfies the mass conservation principle, and is nonnegative for all $t>0$.
%\end{theorem}
%We recall that from~\eqref{hyp:h1} we have that for any $x\in\R$,
%\begin{equation}
%\label{eq:aux0}
%xf(x)\leq xf(0)+C_0x^2-C_0\frac{x^4}{3}.
%\end{equation}

%For the function $\alpha$, it implies assumption~\eqref{hyp:h2} with $C_\ell=C_1$ and $C_r=C_2$.
%$$
%\lim_{y\rightarrow+\infty}\frac{\alpha(y)}{y} = \lim_{y\rightarrow+\infty}\alpha'(y)=C_2,
%$$
%and the same remains true for the limit at $-\infty$. It follows that is also valid.
\begin{proof}[Proof of Theorem~\ref{th:MR-existence}]
The proof relies on a classical argument exhibiting a contraction map whose fixed points are the solutions of the equation. We proceed in two steps: first, introduce a well-suited functional space and a functional to initiate the fixed-point approach, and then analyze the fixed points of said functional. 
% is a classical application of the contraction mapping theorem.
% We divide the proof into two parts. First, we consider a particular map and show that it is well defined, and secondly, we show that at short time intervals, the map is a contraction. The conclusion is a consequence of a fixed-point argument in short times, and iterating we pass to the full interval $[0,T]$.
\smallskip

\noindent\textbf{First step:} Let $\tilde K>0$ be the upper-bound of the family $I_{\mu_\eps}(t)$ found in Lemma~\ref{lem:y2}. Fix a time horizon $T>0$, a constant $\kappa\in(0,1)$, and take $A_\kappa$ as the following closed set:
$$
A_\kappa = \left\{u\in C^1([0,T],L^2_{m_\kappa})\text{ such that }\|u\|_{L^1}=1,\,\,u\geq0,\,\,I_u\leq \tilde K,\,\,\|u(t,\cdot)\|_{L^2_{m_\kappa}}\leq a\right\},
$$
with $a>0$ to be chosen later. Let $\Phi:A_\kappa\rightarrow A_\kappa$ be the application that takes $u\in A_\kappa$ and associates $v$ the weak solution to the linear equation:
$$
\begin{cases}
\partial_tv(t,x) = -\partial_x\left[\left(f(x)-\eps^{-1}\alpha(x)I_u \right)v(t,x)\right]+\frac{\sigma^2}{2}\partial^2_{xx}v(t,x)
\\
v(0,x) = \mu_{0,\eps}(x).
\end{cases}
$$
Since the equation is in divergence form, the mass conservation principle holds, and it is easy to see that any solution is nonnegative. To show that $\Phi$ maps $A_\kappa$ onto itself, we are thus left to checking the validity of the final three inequalities in the definition of $A_\kappa$. First, lemma~\ref{lem:y2} ensures that $I_u$ is bounded. Moreover, we have
\begin{eqnarray*}
\frac12\frac{d}{dt}\int v^2{m_\kappa}  
% &=&-\int\partial_x\left[\left(f(x)-\eps^{-1}\alpha(x)I_u \right)v\right]\cdot vm+\frac{\sigma^2}{2}\int\partial^2_{xx}v\cdot vm
%\\
%&=&-\int \left(f'(x)-\eps^{-1}\alpha'(x)I_u \right)v^2\cdot m-\int \left(f(x)-\eps^{-1}\alpha(x)I_u \right)\partial_xv\cdot vm
%\\
%&&\quad-\frac{\sigma^2}{2}\int |\partial_{x}v|^2\cdot m-\frac{\sigma^2}{2}\int\partial_{x}v\cdot v\partial_xm
%\\
%&=&-\int \left(f'(x)-\eps^{-1}\alpha'(x)I_u \right)v^2\cdot m-\frac12\int \left(f(x)-\eps^{-1}\alpha(x)I_u \right)\partial_xv^2\cdot m
%\\
%&&\quad-\frac{\sigma^2}{2}\int |\partial_{x}v|^2\cdot m-\frac{\sigma^2}{4}\int\partial_{x}v^2\cdot \partial_xm
%\\
%&=&-\int \left(f'(x)-\eps^{-1}\alpha'(x)I_u \right)v^2\cdot m+\frac12\int \left(f'(x)-\eps^{-1}\alpha'(x)I_u \right)v^2\cdot m
%\\
%&&\quad+\frac12\int \left(f(x)-\eps^{-1}\alpha(x)I_u \right)v^2\cdot \partial_xm-\frac{\sigma^2}{2}\int |\partial_{x}v|^2\cdot m+\frac{\sigma^2}{4}\int v^2\cdot \partial^2_{xx}m
%\\
&=&-\frac12\int \left(f'(x)-\frac{\alpha'(x)}{\eps}I_u \right)v^2\cdot {m_\kappa}+\frac12\int \left(f(x)-\frac{\alpha(x)}{\eps}I_u \right)v^2 \partial_x{m_\kappa}\\
&&\quad-\frac{\sigma^2}{2}\int |\partial_{x}v|^2\cdot {m_\kappa}+\frac{\sigma^2}{4}\int v^2 \partial^2_{xx}{m_\kappa}\\
&=& \frac12\int \left(-f'(x)+\frac{\alpha'(x)}{\eps}I_u \right)v^2\cdot {m_\kappa}+\frac12\int \left(2\kappa x f(x)-2\kappa x \frac{\alpha(x)}{\eps}I_u \right)v^2  m_\kappa\\
&&\quad-\frac{\sigma^2}{2}\int |\partial_{x}v|^2\cdot {m_\kappa}+\frac{\sigma^2}{4}\int v^2 \kappa\sigma^2(1+2\kappa x^2) {m_\kappa}\\
&=& \frac12\int \left(2\kappa x f(x)-f'(x)+ \kappa\sigma^2(1+2\kappa x^2) \right)v^2\cdot {m_\kappa}\\
&& \qquad +\frac12\int \left(\alpha'(x)-2\kappa x\alpha(x)\right)\frac{I_u}{\eps} v^2  m_\kappa-\frac{\sigma^2}{2}\int |\partial_{x}v|^2\cdot {m_\kappa}.
%
%\\
%&=&  \int\left[\left(f(x)-\eps^{-1}\alpha(x)I_u \right)v\right]\partial_x[(1+x^2)v] -\frac{\sigma^2}{2}\int\partial_{x}v(t,x)\partial_x[(1+x^2)v]
%\\
%&\leq&  \int\left[\left(f(x)-\eps^{-1}\alpha(x)I_u \right)v\right](2xv+(1+x^2)\partial_xv) -\frac{\sigma^2}{2}\int2x\,v\,\partial_{x}v
%\\
%&=&  \int2x\left(f(x)-\eps^{-1}\alpha(x)I_u \right)v^2 
%\\
%&&\qquad+\frac12 \int (1+x^2)\left(f(x)-\eps^{-1}\alpha(x)I_u \right)\partial_x|v|^2 +\frac{\sigma^2}{2}\int v^2
%\\
%%&=&  \int x\left(f(x)-\eps^{-1}\alpha(x)I_u \right)v^2 
%%\\
%%&&\qquad -\frac12 \int (1+x^2)\left(f'(x)-\eps^{-1}\alpha'(x)I_u \right)|v|^2+\frac{\sigma^2}{2}\int v^2,
\end{eqnarray*}
We now proceed to show that hypotheses~\eqref{hyp:mr-existence} imply that the functions
$$
2\kappa xf(x)-f'(x)+\kappa\sigma^2(1+2\kappa x^2)\qquad\text{and}\qquad \alpha'(x)-2\kappa x\alpha(x)
$$
are upper bounded. Indeed, the first quantity is bounded as a direct consequence of hypothesis~\eqref{hyp:h1}, which indeed implies that for any $x\in\R$,
\begin{equation}
\label{eq:aux0}
xf(x)\leq xf(0)+C_0x^2-C_0\frac{x^4}{3}.
\end{equation}
To realize that $ \alpha'(x)-2\kappa x\alpha(x)$ is upper-bounded, we start by noting that since $\alpha'(x) \to C_2$ at $\infty$, there exists $M_+>0$ large enough such that, for any $x>M_+$, $\frac{C_2}{2}<\alpha'(x)<\frac{3C_2}{2}$, thus implying a linear bound on $\alpha$:
$$
\alpha(x)=\alpha(M_+)+\int_{M_+}^x \alpha'(y)dy\geq C'+C''x,
$$
for some $C',C''$ constants with $C''>0$, thus implying that:
$$
\alpha'(x)-2\kappa x\alpha(x)\leq \frac{3C_2}{2}-2\kappa x(C'+C''x),
$$
which is an upper-bounded function.  A similar argument applies on the interval $(-\infty,-M_-)$ with $M_->0$ large enough. Finally, on $[-M_-,M_+]$ the expression $\alpha'-2\kappa \alpha$ is continuous, thus also bounded. Altogether, these three remarks ensure that $\alpha'-2\kappa \alpha$ is bounded on $\R$. 

By assumption, $I_u$ is bounded; it is also non-negative since $\beta$ is. We deduce the existence of a constant $C>0$ (depending on $\tilde K$, itself depending on the uniform bound of the initial conditions $K_0$, but not on $\eps$) such that
$$
\frac{d}{dt}\int v^2\cdot {m_\kappa} \leq C(1+\eps^{-1})\int v^2\cdot {m_\kappa}.
$$
Gronwall's lemma implies that there is a positive constant $K$ such that:
\begin{equation}
\label{eq:aux2}
\|v(t,x)\|_{L^2_{m_\kappa}} \leq \left(\int |\mu_\eps(0,\cdot)|^2\cdot {m_1}\,dx \right)^{1/2}e^{C(1+\eps^{-1})T/2}=\sqrt{K_0}e^{C(1+\eps^{-1})T/2},
\end{equation}
an upper bound that is independent of $u$. Setting $a=\sqrt{K_0}e^{C(1+\eps^{-1})T/2}$, we have that that last inequality in the definition of $A_\kappa$ also holds true.

Finally, since all hypotheses of Lemma~\ref{lem:y2} are valid, we can replicate the arguments to find that $I_v$ is also upper-bounded by $\tilde K$. From here, we deduce that the map $\Phi$ is well-defined. 
\smallskip

\noindent\textbf{Second step:} Before showing that the map is a contraction for $T>0$ small enough, we notice that from the left-hand side of~\eqref{eq:aux2}, if we take $\kappa=1$ it follows the existence of $a_1$ such that for any $0<\kappa<1$ we have that
$$
\|v(t,x)\|_{L^2_{m_\kappa}}=\int v(t,x)^2e^{\kappa(1+x^2)} \leq \int v(t,x)^2e^{(1+x^2)}  \leq a_1^2.
$$ 

Let $u_1,u_2\in A$ and the corresponding $v_1=\Phi(u_1), v_2=\Phi(u_2)$. For simplicity, when needed we use $I_i(t)=I_{u_i}(t)$ for $i=1,2$. It holds that
\begin{align*}
\frac12\frac{d}{dt}\int(v_1-v_2)^2 {m_\kappa}=&-\frac12\int \left(f'(x)-\eps^{-1}\alpha'(x)I_1 \right)(v_1-v_2)^2\cdot {m_\kappa}
\\
&
+\frac12\int \left(f(x)-\eps^{-1}\alpha(x)I_1 \right)(v_1-v_2)^2\cdot \partial_x{m_\kappa}\\
&-\frac{\sigma^2}{2}\int |\partial_{x}(v_1-v_2)|^2\cdot {m_\kappa}+\frac{\sigma^2}{4}\int (v_1-v_2)^2\cdot \partial^2_{xx}{m_\kappa}
\\
&-\frac{I_1-I_2}{\eps}\int \alpha v_2\partial_x(v_1-v_2)\cdot {m_\kappa}-\frac{I_1-I_2}{\eps}\int\alpha v_2(v_1-v_2)\partial_x{m_\kappa}
\end{align*}
H\"older's inequality yields:
$$
|I_1-I_2| \leq \left(\int\frac{\beta^2(x)}{m_\kappa}\right)^{1/2}\left(\int (u_1-u_2)^2\cdot {m_\kappa}\right)^{1/2}\leq C\left(\int (u_1-u_2)^2\cdot {m_\kappa}\right)^{1/2},
$$
because $\beta$ has at most polynomial growth at $\pm\infty$. Now we take advantage of $\kappa$ to get the conclusion. More precisely, we have that
\begin{align*}
\int \alpha(x)v_2\partial_x(v_1-v_2)\cdot{m_\kappa}&\leq \left(\int \alpha^2v_2^2m_\kappa\right)^{1/2}\left(\int |\partial_x(v_1-v_2)|^2\cdot {m_\kappa}\right)^{1/2}
\\
&\leq C \left(\int v_2^2\cdot m_1\right)^{1/2}\left(\int |\partial_x(v_1-v_2)|^2\cdot {m_\kappa}\right)^{1/2}
\\
&\leq Ca_1\left(\int |\partial_x(v_1-v_2)|^2\cdot {m_\kappa}\right)^{1/2},
\end{align*}
similarly, since $\partial_xm_\kappa=2\kappa xm_\kappa$, then
\begin{align*}
\int \alpha(x)v_2(v_1-v_2)\cdot\partial_x{m_\kappa}&=
\int 2\kappa x\alpha(x)v_2(v_1-v_2)\cdot {m_\kappa}&
\\
&\leq \left(\int 4\kappa^2x^2\alpha^2v_2^2m_\kappa\right)^{1/2}\left(\int (v_1-v_2)^2\cdot {m_\kappa}\right)^{1/2}
%\\
%&\leq C \left(\int v_2^2\cdot m_1\right)^{1/2}\left(\int |\partial_x(v_1-v_2)|^2\cdot {m_\kappa}\right)^{1/2}
\\
&\leq Ca_1\left(\int (v_1-v_2)^2\cdot {m_\kappa}\right)^{1/2}.
\end{align*}
Summarizing, there exists a constant $C$ depending on $\eps$ such that
\begin{align*}
-\eps^{-1}(I_1-I_2)\int \alpha v_2\partial_x(v_1-v_2)\cdot {m_\kappa} &\leq  Ca_1\|u_1-u_2\|_{L_{m_\kappa}^2}\|\partial(v_1-v_2)\|_{L_{m_\kappa}^2}
\\&\leq \frac{4a_1^2C^2}{\sigma^2}\|u_1-u_2\|_{L_{m_\kappa}^2}^2+\frac{\sigma^2}{4}\|\partial(v_1-v_2)\|_{L_{m_\kappa}^2}^2,
\end{align*}
and also
\begin{align*}
-\eps^{-1}(I_1-I_2)\int \alpha v_2(v_1-v_2)\cdot \partial_x{m_\kappa} &\leq  Ca_1\|u_1-u_2\|_{L_{m_\kappa}^2}\|v_1-v_2\|_{L_{m_\kappa}^2}
\\&\leq C^2a_1^2\|u_1-u_2\|_{L_{m_\kappa}^2}^2+\|v_1-v_2\|_{L_{m_\kappa}^2}^2,
\end{align*}
thus
\begin{align*}
\frac12\frac{d}{dt}\|v_1-v_2\|^2_{L_{m_\kappa}^2}&\leq C(1+\eps^{-1})\|v_1-v_2\|_{L_{m_\kappa}^2}^2+C^2a_1^2\left(\frac{4}{\sigma^2}+1\right)\|u_1-u_2\|_{L_{m_\kappa}^2}^2
\\
&= \frac{C_1}2\|v_1-v_2\|_{L_{m_\kappa}^2}^2+\frac{C_2}2\|u_1-u_2\|_{L_{m_\kappa}^2}^2,
\end{align*} 
for well chosen positive constants $C_1(\eps)$ and $C_2$. By Gronwall's lemma, it follows that
$$
\sup_{t\in[0,T]}\|v_1-v_2\|_{L_{m_\kappa}^2}\leq \sqrt{\frac{C_2}{C_1}(e^{C_1T}-1)}\sup_{t\in[0,T]}\|u_1-u_2\|_{L_{m_\kappa}^2}
$$
and finally, for $T$ small enough such that $\sqrt{\frac{C_2}{C_1}\left(e^{C_1T}-1\right)}<1$, $\Phi$ is a contraction. Therefore, $\Phi$ has a unique fixed point and there exists $\mu_\eps(t,x)\in A$ a solution to the following equation
$$
\partial_t\mu(t,x) = -\partial_x\left[\left(f(x)-\eps^{-1}\alpha(x) \int \beta(y)\mu(t,y)\right)\mu(t,x)\right]+\frac{\sigma^2}{2}\partial^2_{xx}\mu(t,x),\,t\in(0,T),
$$
with initial condition $\mu_\eps(0,x)=\mu_{0,\eps}(x)$. Iterating in time, we find a global solution to equation~\eqref{eq:pde}. 
\end{proof}

To finish this section, we show that the notion of convergence of the sequence of functions $I_{\mu_\eps}(t)$ as $\eps$ goes to zero. This will provide us with a specific concentration profile for the interaction function, which will be helpful in proving a limit for the sequence of solutions $\mu_\eps$.

\begin{lemma}
\label{lem:Jeps}
Assume that the hypotheses of Theorem~\ref{th:MR-existence} hold, and that the functions $f,\alpha$ and $\beta$ are such that there exists $C>0$ and a non-negative polynomial $p$ such that:
\begin{equation}
\label{hyp:h4}
C^{-1}\leq \beta'(y)\alpha(y),\qquad \beta''(y)\alpha^2(y)+\beta'(y)\alpha'(y)\alpha(y)\geq0,\qquad \forall y\in\R,
\end{equation}
and
\begin{equation}
\label{hyp:h5}
|f''(y)|\leq C(1+|y|),\qquad |\beta^{(iv)}(y)|\leq C(1+p(y)),\qquad |\alpha''(y)|\leq C(1+p(y)).
\end{equation}
Then, we have the local uniform BV bound
$$
\int_0^T\left| \frac{d}{dt}I_\eps(t)\right|\leq C'+C''T,
$$
where $C'$ and $C''$ are positive constants. Consequently, after extraction of a subsequence, $I_\eps(t)$ converges a.e. to a function $I(t)$, as $\eps$ goes to zero.
\end{lemma}

%Before the proof, we remark that inequalities~\eqref{hyp:h4}-\eqref{hyp:h5} are met by polynomial functions well chosen, for instance
%$\beta(y) = C(1+y+\frac{y^2}{2}),$ $\alpha(y) = C(1+y),$ and $f(y) = C(1-y^3),$
%then
%$$
%\dot\beta(y)\alpha(y) = C^2(1+y)^2,\text{ 
%and }
%\ddot\beta(y)\alpha^2(y)+\dot\beta(y)\dot\alpha(y)\alpha(y)=2C^3(1+y)^2\geq0.
%$$
Before we proceed to proving this result, we recall that for a general function $g(y)$, a simple application of the mean value theorem implies that
$$
\text{if }(\exists C>0,k\in\mathbb N)\text{ such that }|g'(y)|\leq C(1+|y|^k) \text{ then }|g(y)|\leq C'(1+|y|^{k+1}),
$$
for a possibly larger constant $C'$. This means that under~\eqref{hyp:h5} functions $f,\alpha$ along with their first and second derivatives are upper-bounded by polynomial functions. The same remark holds for $\beta$ and its first four derivatives.

\begin{proof}
We have that
\begin{eqnarray}
\nonumber
\frac{d}{dt}I_\eps(t) 
%&=& \int_y \dot\beta(y)\Big(f(y)-\eps^{-1}\int_zb(y,z)\mu_\eps(t,z)\Big)\mu_\eps(t,dy)+\frac{\sigma^2}{2}\int_y\ddot\beta(y)\mu_\eps(t,dy)
%\\
\nonumber
&=& \int \beta'(y)\Big(f(y)-\eps^{-1}\alpha(y)I_{\mu_\eps}(t)\Big)\mu_\eps(t,dy)+\frac{\sigma^2}{2}\int\beta''(y)\mu_\eps(t,dy)
\\
&=& \int \underbrace{\Big(\beta'(y)f(y)+\frac{\sigma^2}{2}\beta''(y)\Big)}_{g_1(y)}\mu_\eps(t,dy)-\frac{I_{\mu_\eps}(t)}{\eps}\int\underbrace{\beta'(y)\alpha(y)}_{g_2(y)}\mu_\eps(t,dy)
\label{eq:ineq_aux}\\
&\leq &\int g_1(y)\mu_\eps(t,dy)\;\;\;\leq\;\;\; \int C'(1+q(y))\mu_\eps(t,dy),\nonumber
\end{eqnarray}
 where $q$ is some polynomial in $y$, using the assumption that $\beta^{(iv)}$ and $f$ are sub-polynomial. This expression provides a uniform bound to the left-hand side, owing to the bounds of Gaussian moments of $\mu_\eps$ proven in Theorem~\ref{th:MR-existence}.

Let us now establish the lower-bound inequality. We rewrite the term $J_\eps(t):=\frac{d}{dt}I_\eps(t)$ as follows:
\begin{eqnarray*}
\frac{d}{dt}J_\eps(t)
&=& \int g_1'(y)\Big(f(y)-\eps^{-1}\alpha(y)I_{\mu_\eps}(t)\Big)\mu_\eps(t,dy)+\frac{\sigma^2}{2}\int g_1''(y)\mu_\eps(t,dy)
\\
&&\quad-\frac{I_{\mu_\eps}(t)}{\eps}\int g_2'(y)\Big(f(y)-\eps^{-1}\alpha(y)I_{\mu_\eps}(t)\Big)\mu_\eps(t,dy)
\\
&&\quad\quad -\frac{I_{\mu_\eps}(t)}{\eps}\frac{\sigma^2}{2}\int g_2''(y)\mu_\eps(t,dy)
-\frac{J_\eps(t)}{\eps}\int g_2(y)\mu_\eps(t,dy). 
\end{eqnarray*}
We notice that the term proportional to $\eps^{-2}$, $I_{\mu_\eps}(t)\int g_2'(y)\alpha(y)I_{\mu_\eps}(t)\mu_\eps(t,dy)$, is non-negative (and can thus be ignored for establishing a lower-bound), since:
$$
\int g_2'(y)\alpha(y)\mu_\eps(t,dy) =  \int\left(\beta''(y)\alpha^2(y)+\beta'(y)\alpha'(y)\alpha(y)\right)\mu_\eps(t,dy)\geq0
$$
under assumption~\eqref{hyp:h4}. 

The rest of the proof relies on the assumed sub-polynomial growth of the functions $f,\alpha$ and $\beta$, along with Lemma~\ref{lem:y2} and Gaussian moments of the solutions. All terms follow the same approach, which we make explicit for the same term for clarity. 
$$
\int  g_1'(y)f(y)\mu_\eps(t,dy) = \int \left(\beta''(y)f(y)+\beta'(y)f'(y)+\frac{\sigma^2}{2}\beta'''(y)\right)f(y)\mu_\eps(t,dy)
$$
and there is a constant positive $C'$ and a nonnegative polynomial $q(y)$ such that
$$
\left|\beta''(y)f(y)+\beta'(y)f'(y)+\frac{\sigma^2}{2}\beta'''(y)\right|\leq C'(1+q(y)),
$$
readily implying
$$
\int g_1'(y)f(y)\mu_\eps(t,dy)\geq -C'\int(1+q(y))\mu_\eps(t,dy)\geq -C''
$$
thanks to the uniform bounds on the Gaussian moment of $\mu_\eps$. Other terms are readily treated using a similar approach,
%For the second term, a very similar calculation leads us to
%$$
%|\dot g_1(y)\alpha(y) I_{\mu_\eps}(t)|\leq C'(1+|y|^5)
%$$
%and as a consequence we have the following bound
%$$
%\int_y \dot g_1(y)\Big(f(y)-\eps^{-1}\alpha(y)I_{\mu_\eps}(t)\Big)\mu_\eps(t,dy)\geq - \frac{C''}{\eps},
%$$
%with positive $C''$ adequately chosen. The second term involves the second derivative of $g_1$ which is
%$$
%\ddot g_1(y)=\dddot\beta(y)f(y)+2\ddot\beta(y)\dot f(y)+\dot\beta(y)\ddot f(y)+\frac{\sigma^2}{2}\ddddot\beta(y)
%$$
%and again there are some constants $C'$ and $C''$ such that
%$$
%|\ddot g_1(y)|\leq C'(1+|y|^4)\quad\Rightarrow\quad \frac{\sigma^2}{2}\int_y\ddot g_1(y)\mu_\eps(t,dy)\geq - C'',
%$$
% thanks to the uniform bounds of the moments of $\mu_\eps$. For the next two terms, a similar argument implies the existence of $C''>0$ such that
% $$
%- \frac{I_{\mu_\eps}(t)}{\eps}\int_y \dot g_2(y) f(y)\mu_\eps(t,dy)-\frac{I_{\mu_\eps}(t)}{\eps}\frac{\sigma^2}{2}\int_y\ddot g_2(y)\mu_\eps(t,dy)\geq -\frac{C''}{\eps},
% $$
% because $\alpha$ and $f$ have similar growing bounds according to~\eqref{hyp:h5}. 
%Summarising, we have the following bound for the second time derivative of $I_\eps(t)$:
ensuring the following bound:
 $$
 \frac{d}{dt}J_\eps(t)\geq -\frac{C''}{\eps}-\frac{J_\eps(t)}{\eps}\int g_2(y)\mu_\eps(t,dy),
 $$
 for a possibly larger constant $C''>0$. Finally, take $(J_\eps(t))_- = \max(0,-J_\eps(t))$, we show that for $J_\eps(t)<0$:
$$
(J_\eps(t))'_- =-J_\eps(t)'\leq \frac{C''}{\eps}-\frac{-J_\eps(t)}{\eps}\int g_2(y)\mu_\eps(t,dy)= \frac{C''}{\eps}-\frac{(J_\eps(t))_-}{\eps}\int g_2(y)\mu_\eps(t,dy),
$$
and from the hypothesis~\eqref{hyp:h4}, it follows that
$$
(J_\eps(t))'_-\leq\frac{C''}{\eps}-\frac{1}{C\eps}(J_\eps(t))_-.
$$
The previous inequality is, in fact, trivial for $J_\eps(t)>0$, since in that case the derivative is zero. From this inequality, it is standard to deduce that
$$
-J_\eps(t)\leq\max(0,-J_\eps(t))=(J_\eps(t))_-\leq C''C+(J_\eps(0))_-e^{-t/C\eps},
$$
from we get that $J_\eps(t)$ is uniformly lower-bounded.
%Take now $(J_\eps(t))_+=\max(0,J_\eps(t))$, thus whenever $J_\eps(t)>0$:
%$$
%(J_\eps(t))'_+=J_\eps(t)'\geq -\frac{C''}{\eps}-\frac{(J_\eps(t))_+}{\eps}\int_yg_2(y)\mu_\eps(t,dy)\geq -\frac{C''}{\eps}-\frac{C_2''}{\eps}(J_\eps(t))_+,
%$$
%because $|g_2(y)|\leq C'(1+|y|^6)$, and as a consequence
%$$
%(J_\eps(t))_+\geq -\frac{C''}{C''_2}+(J_\eps(0))_+e^{-tC_2''/\eps}.
%$$
We finish the proof by noticing that
\begin{align*}
\int_0^T\left|\frac{d}{dt}I_\eps(t)\right|\,dt &=\int_0^T \frac{d}{dt}I_\eps(t)\,dt +2\int_0^T\left(\frac{d}{dt}I_\eps(t)\right)_-\,dt 
\\
&\leq I_\eps(T)-I_\eps(0)+2\left(C''CT+(J_\eps(0))_-\int_0^Te^{-t/C\eps}\,dt\right)
\\
&=I_\eps(T)-I_\eps(0)+2\left(C''CT+(J_\eps(0))_-C\eps (1-e^{-T/C\eps})\right),
\end{align*}
which implies that there are some positive constants $C',C''$ such that
$$
\int_0^T\left|\frac{d}{dt}I_\eps(t)\right|\,dt \leq C'T+C''.
$$
$\,$
\end{proof}

\section{Hopf-Cole transformation analysis}
\label{sec:hopf}

We have established so far that, under the hypotheses of Theorem~\ref{th:MR-existence}, we have that for fixed $\kappa\in(0,1)$, for each $\eps>0$ there is a probability density $\mu_{\eps}\in C([0,T],L^2_\kappa(\R))$ that is a weak solution of equation~\eqref{eq:pde}. As a consequence, at least formally, we can take the Hopf-Cole transformation $\phi_\eps(t,x)=\eps\ln\mu_\eps(t,x)$ and observe that
$$
\partial_t\mu_\eps(t,x) =\frac{1}{\eps}\mu_\eps(t,x)\partial_t\phi_\eps(t,x),\quad \partial_x\mu_\eps(t,x) = \frac{\mu_\eps(t,x)}{\eps}\partial_x\phi_\eps(t,x),
$$
and
$$
\partial^2_{xx}\mu_\eps(t,x)= \frac1{\eps^2}\mu_\eps(t,x)|\partial_x\phi_\eps(t,x)|^2+\frac1{\eps}\mu_\eps(t,x)\partial^2_{xx}\phi_\eps(t,x).
$$
therefore $\phi_\eps$ is a solution of the equation:
\begin{multline}
\partial_t\phi_\eps = -\left(\eps f'(x) - \alpha'(x)I_{\mu_\eps}(t) \right)
%\\
-  \frac{\eps f(x)- \alpha(x) I_{\mu_\eps}(t)}{\eps}\partial_x\phi_\eps + \frac{\sigma^2}{2\eps} \vert \partial_x \phi_\eps \vert^2  + \frac {\sigma^2}{2} \partial_{xx}^2 \phi_\eps.
\label{eq:hopfcole}
\end{multline}
%Now, consider only terms with $\eps^{-1}$ to find that $\phi_\eps$ is a solution to
%\[
%0=\frac{\sigma^2}{2} \vert \partial_x \phi_\eps \vert^2 -\alpha(x)I_{\mu_\eps}(t) \partial_x \phi_\eps
%\quad
%\text{or}
%\quad
%\frac{\sigma^2}{2} \partial_x \phi =\alpha(x)I_{\mu_\eps}(t).\]
%In the particular case $b(x,y)=x-y$, 
%\[\frac{\sigma^2}{2} \partial_x \phi =x-\bar{x}\]
%so we would have $\phi=C(x-\bar{x})^2$. 
%
%If $\phi$ has a single $0$ at $x^*$ and is strictly negative otherwise (e.g., concave taking its max, 0, at $x^*$), 
%\[\frac{\sigma^2}{2} \partial_x \phi =b(x,x^*)\]
%\[\phi =C+\frac{2}{\sigma^2} \int_{0}^{x} b(x,x^*) \]
%with $C$ such that the max of the function is $0$. The concavity assumption is true if, for all $x$,
%\[\partial_x^2 \phi =\frac{2}{\sigma^2} \partial_x b(x,x^*)<0\]
%

The function $\phi_\eps$ is thus in the kernel of the following operator acting on maps continuously differentiable with respect to time and twice continuously differentiable with respect to space $C^1([0,\infty);C^2(\R))$:
\begin{multline*}
\mathcal L_\eps\phi:=\partial_t\phi +\left(\eps f'(x) - \alpha'(x)I_{\mu_\eps}(t) \right)
\\
+  \left(\eps f(x)- \alpha(x) I_{\mu_\eps}(t)\right)\frac{1}{\eps}\partial_x\phi - \frac{\sigma^2}{2\eps} \vert \partial_x \phi \vert^2  - \frac {\sigma^2}{2} \partial_{xx}^2 \phi.
\end{multline*}
For simplicity we write $I_\eps=I_{\mu_\eps}$, and $\Lambda(x)=\int_0^x\alpha(s)ds$. 

\begin{proposition}
\label{prop:supersol}
Assume that initial conditions $\phi_\eps(0,x)=\eps\log\mu_{\eps,0}(x)$ satisfy the conditions
$$
\sup_{0<\eps<1}\eps\log \mu_{0,\eps}(x) \leq -Ax^2+B,
$$
for some constants $A,B$ with $A$ positive. Then there is another set of positive constants $A',B',D'$ and $E'$ such that
$$
\phi_\eps(t,x)\leq -A'I_\eps(t)\Lambda(x)-\frac{\eps}2 B'x^2+D't+E',
$$ 
at all positive times.
\end{proposition}

\begin{proof}
Consider the map
$$
\psi_\eps(t,x) = -A'I_\eps(t)\Lambda(x)-\frac{\eps}2 B'x^2+D't+E',
$$
for some $A',B',D'$ and $E'$ positive constants to be specified later. Applying the map $\mathcal L_\eps$ to $\psi_\eps(t,x)$ we find that
\begin{multline*}
\mathcal L_\eps\psi_\eps = D'-A'I'_\eps\Lambda(x)+\left(\eps f'(x) - \alpha'(x)I_{\eps}(t) \right)
\\
-  \left(\eps f(x)- \alpha(x) I_{\eps}(t)\right)\frac{1}{\eps}\Big(\eps B'x+A'I_\eps(t)\alpha(x)\Big)-\frac{\sigma^2}{2\eps}\Big(\eps B'x+A'I_\eps(t)\alpha(x)\Big)^2
\\
+\frac{\sigma^2}{2}\Big(\eps B'+A'I_\eps(t)\alpha'(x)\Big)=:D'+\eps^{-1} h_1(t,x)+h_2(t,x)+\eps h_3(t,x),
\end{multline*}
with
\begin{align*}
h_1(t,x) &= A'\alpha(x)^2I_\eps(t)^2\left(1-\frac{A'\sigma^2}{2}\right),
\\
h_2(t,x) &= -A'I'_\eps(t)\Lambda(x)+\alpha'(x)I_\eps(t)\left(\frac{A'\sigma^2}{2}-1\right)
\\
&+B'x\alpha(x)I_\eps(t)(1-\sigma^2 A') -f(x)A'I_\eps(t)\alpha(x),
\\
h_3(t,x) &= f'(x) - B'xf(x) - B'^2x^2\frac{\sigma^2}{2}+ B'\frac{\sigma^2}{2}.
\end{align*}
First, we work with the terms of order $\eps^{-1}$, to find that for any $A'<\frac{2}{\sigma^2}$ it holds that
$$
h_1(t,x)= A\left(1-\frac{\sigma^2 A}{2}\right)I_\eps(t)^2\alpha(x)^2\geq0.
$$
For the other terms, we set $M>0$ large and notice that $h_2(t,\cdot)$ and $h_3(t,\cdot)$ are continuous functions, and then there is a positive constant $C_M$ such that they are larger than $-C_M$ for all $x\in[-M,M]$. Secondly, from~\eqref{eq:aux0} it follows that
\begin{align*}
h_3(t,x) &= f'(x) - B'xf(x) - (B')^2x^2\frac{\sigma^2}{2}+ B'\frac{\sigma^2}{2}
\\
&\geq  - C'(1+x^2)- B'\left(xf(0)+C_0x^2-C_0\frac{x^4}3\right)\geq -C'+C''x^4,
\end{align*}
with $C',C''$ two positive generic constants, $C''$ small enough. Finally, for $h_2(t,x)$, we use the fact that both $I_\eps$ and $I'_\eps$ are uniformly bounded. Again, $h_2$ is continuous and bounded in the interval $[-M,M]$. Outside this region, have that $\Lambda(x)=O(1+x^2)$, $\alpha'(x)=O(1)$. Moreover $x\alpha(x)=O(1+x^2)$ and $f(x)\alpha(x)=(xf(x))\frac{\alpha(x)}{x}=O(1+x^2)$, then it follows that
$$
h_2(t,x)=O(1+x^2)\geq -C'(1+x^2)
$$
for some $C'>0$.
Altogether, we obtain that:
\begin{align*}
\mathcal L_\eps\psi_\eps &= D'+\eps^{-1}h_1(t,x)+h_2(t,x)+\eps h_3(t,x)
\\
&\geq D'-C'(1+x^2)-C'\eps+C''x^4\eps
\end{align*}
which is nonnegative for $D'=D'(\eps)$ large enough. To finish the proof, notice that from the uniform initial conditions bound, we have that
$$
\sup_{0<\eps<1}\phi_\eps(0,x)=\sup_{0<\eps<1}\eps\log \mu_{0,\eps}(x) \leq -Ax^2+B,
$$
but since $I_\eps(t)$ is uniformly bounded by some $C_I$ it follows that 
$$
\psi_\eps(0,x)= -A'I_\eps(0)\Lambda(x)-\frac{\eps}2 B'x^2+E' \geq -A'C_I|\Lambda(x)|-\frac{\eps}2 B'x^2+E',
$$
and again $\Lambda(x)=O(1+x^2)$ outside of the compact $[-M,M]$ and bounded by below inside $[-M,M]$. As a consequence, there is a constant $C>0$ such that
$$
\psi_\eps(t,x)\geq -C(1+x^2)-\frac{\eps}2 B'x^2+E'\geq -Ax^2+B.
$$
This proves that the inequality claimed in the proposition is true for the initial condition, for a proper choice of constants.
\end{proof}

\subsection{Regularity of the solutions of the problem}

We now prove that the nonlinear operator $\mathcal L_\eps$ has a regularizing effect such that solutions are uniformly Lipschitz in space at any positive time, independently of the regularity of the initial conditions. To that aim, we will study the equation associated to with the space derivative of $\phi_\eps(t,x)$. In particular, we have the following:
\begin{proposition}
\label{prop:regularity_space}
Assume that $\eps\leq 1$ and let $F=\sqrt{E'+D'T}$ for $E'$ and $D'$ the same positive constants from Proposition~\ref{prop:supersol}, and let $T>0$ be an arbitrary time horizon. Then the map $w_\eps(t,x)=\sqrt{2F^2-\phi_\eps(t,x)}$ is well defined and there exists a constant $\theta(T)$independent of $\eps$ such that
$$
\left|\partial_x w_\eps(t,x)\right|\leq  \sqrt{\frac{1}{t\sigma^2}}+\theta(T).
$$
\end{proposition}

\begin{proof}
First, notice that proposition~\ref{prop:supersol} implies in particular that
$$
\phi_\eps(t,x)\leq D'T+E'\quad\Rightarrow\quad 2F^2-\phi_\eps(t,x)\geq D'T+E'>0,
$$
and thus ensures that $w_\eps$ is well-defined. Our proof of an upper bound relies on a well-chosen change of variable. In detail, let $\gamma$ be a smooth invertible map, with nonvanishing derivative, and define $\phi_\eps(t,x)=\gamma(w_\eps(t,x))$. Elementary calculus yields:
$$
\partial_t\phi_\eps = \gamma'(w_\eps)\partial_tw_\eps,\qquad \partial_x\phi_\eps = \gamma'(w_\eps)\partial_xw_\eps,
$$
and
$$
\partial^2_{xx}\phi_\eps=\gamma''(w_\eps)|\partial_xw_\eps|^2+\gamma'(w_\eps)\partial^2_{xx}w_\eps,
$$
implying that $w_\eps$ is a solution to
\begin{multline}
\partial_tw_\eps = -\frac{\left(\eps f'(x) - \alpha'(x)I_{\mu_\eps}(t) \right)}{\gamma'(w_\eps)}-  \left( f(x)- \eps^{-1}\alpha(x) I_{\mu_\eps}(t)\right)\partial_xw_\eps
\\
 + \frac{\sigma^2}{2}\left(\frac{\gamma'(w_\eps)}{\eps}+\frac{\gamma''(w_\eps)}{\gamma'(w_\eps)}\right) \vert \partial_x w_\eps \vert^2  + \frac {\sigma^2}{2} \partial_{xx}^2 w_\eps.
\label{eq:wepsilon}
\end{multline}
Define $p_\eps(t,x):=\partial_xw_\eps$, and take an extra spatial partial derivative on~\eqref{eq:wepsilon} to get
\begin{multline*}
\partial_tp_\eps- \frac {\sigma^2}{2} \partial_{xx}^2 p_\eps = -\frac{\left(\eps f''(x) - \alpha''(x)I_{\mu_\eps}(t) \right)}{\gamma'(w_\eps)}+\frac{\left(\eps f'(x) - \alpha'(x)I_{\mu_\eps}(t) \right)}{|\gamma'(w_\eps)|^2}\gamma''(w_\eps)\,p_\eps
\\
-  \left( f'(x)- \eps^{-1}\alpha'(x) I_{\mu_\eps}(t)\right)p_\eps-  \left( f(x)- \eps^{-1}\alpha(x) I_{\mu_\eps}(t)\right)\partial_xp_\eps
\\
 + \frac{\sigma^2}{2}\left(\frac{\gamma''(w_\eps)}{\eps}+\frac{\gamma'''(w_\eps)}{\gamma'(w_\eps)}-\frac{\gamma''(w_\eps)^2}{\gamma'(w_\eps)^2}\right)  p_\eps^3
 \\ 
 + \sigma^2\left(\frac{\gamma'(w_\eps)}{\eps}+\frac{\gamma''(w_\eps)}{\gamma'(w_\eps)}\right) p_\eps \cdot\partial_x p_\eps,
\end{multline*}
and multiplying by $p_\eps$ the equation becomes
\begin{multline*}
\partial_t\frac{|p_\eps|^2}{2}-\frac{\sigma^2}{2}|\partial_x|p_\eps||^2-\frac{\sigma^2}{2}|p_\eps|\partial^2_{xx}|p_\eps|+\frac{\sigma^2}{2}|\partial_xp_\eps|^2
=
\\
-\frac{\left(\eps f''(x) - \alpha''(x)I_{\mu_\eps}(t) \right)}{\gamma'(w_\eps)}\, p_\eps+\frac{\left(\eps f'(x) - \alpha'(x)I_{\mu_\eps}(t) \right)}{|\gamma'(w_\eps)|^2}\gamma''(w_\eps)\,|p_\eps|^2
\\
-  \left( f'(x)- \eps^{-1}\alpha'(x) I_{\mu_\eps}(t)\right)\,|p_\eps|^2-  \frac12\left( f(x)- \eps^{-1}\alpha(x) I_{\mu_\eps}(t)\right)\partial_x|p_\eps|^2
\\
 + \frac{\sigma^2}{2}\left(\frac{\gamma''(w_\eps)}{\eps}+\frac{\gamma'''(w_\eps)}{\gamma'(w_\eps)}-\frac{\gamma''(w_\eps)^2}{\gamma'(w_\eps)^2}\right) |p_\eps|^4
 \\ 
 + \frac{\sigma^2}{2}\left(\frac{\gamma'(w_\eps)}{\eps}+\frac{\gamma''(w_\eps)}{\gamma'(w_\eps)}\right) p_\eps \cdot\partial_x |p_\eps|^2.
\end{multline*}
We find that:
\begin{multline*}
\partial_t|p_\eps|-\frac{\sigma^2}{2}\partial^2_{xx}|p_\eps|
\leq 
\\
-\frac{\left(\eps f''(x) - \alpha''(x)I_{\mu_\eps}(t) \right)}{\gamma'(w_\eps)}\, \frac{p_\eps}{|p_\eps|}+\frac{\left(\eps f'(x) - \alpha'(x)I_{\mu_\eps}(t) \right)}{|\gamma'(w_\eps)|^2}\gamma''(w_\eps)\,|p_\eps|
\\
-  \left( f'(x)- \eps^{-1}\alpha'(x) I_{\mu_\eps}(t)\right)|p_\eps|-  \left( f(x)- \eps^{-1}\alpha(x) I_{\mu_\eps}(t)\right)\partial_x|p_\eps|
\\
 + \frac{\sigma^2}{2}\left(\frac{\gamma''(w_\eps)}{\eps}+\frac{\gamma'''(w_\eps)}{\gamma'(w_\eps)}-\frac{\gamma''(w_\eps)^2}{\gamma'(w_\eps)^2}\right) |p_\eps|^3
 \\ 
 + \sigma^2\left(\frac{\gamma'(w_\eps)}{\eps}+\frac{\gamma''(w_\eps)}{\gamma'(w_\eps)}\right) p_\eps \cdot\partial_x |p_\eps|.
\end{multline*}
Let us now specify the map $\gamma: x\mapsto - x^2+2F(T)^2$. Recall that, by definition, we have $F(T)\leq w_\eps$, and, moreover
$$
\frac{\gamma''(w_\eps)}{\eps}+\frac{\gamma'''(w_\eps)}{\gamma'(w_\eps)}-\frac{\gamma''(w_\eps)^2}{\gamma'(w_\eps)^2} \leq -\frac{2}{\eps}-\frac{1}{w_\eps^2}\leq -2\eps^{-1}.
$$
We also recall that $f$ and $\alpha$ have linear growth at infinity, thus all their derivatives are bounded. Altogether, we have that:
\begin{multline*}
\partial_t|p_\eps|-\frac{\sigma^2}{2}\partial^2_{xx}|p_\eps|
\leq \frac{C_M(1+\eps)}{2F(T)}+\frac{C_M(1+\eps)}{2F(T)^2}|p_\eps|
+C_M(1+\eps^{-1})|p_\eps|- \sigma^2\eps^{-1} |p_\eps|^3
\\
-  \left( f(x)- \eps^{-1}\alpha(x) I_{\mu_\eps}(t)\right)\partial_x|p_\eps|
 + \sigma^2\left(\frac{\gamma'(w_\eps)}{\eps}+\frac{\gamma''(w_\eps)}{\gamma'(w_\eps)}\right) p_\eps \cdot\partial_x |p_\eps|.
\end{multline*}
Assuming that $\eps<1$ (consistent with the large connectivity regime we are interested in), there exists a constant $C_M(T)$ such that:
\begin{multline*}
\partial_t|p_\eps|-\frac{\sigma^2}{2}\partial^2_{xx}|p_\eps|
\leq C_M(T)(1+|p_\eps|)
- \frac{\sigma^2}2|p_\eps|^3
+\eps^{-1}\left(C_M|p_\eps|- \frac{\sigma^2}2 |p_\eps|^3\right)
\\
-  \left( f(x)- \eps^{-1}\alpha(x) I_{\mu_\eps}(t)\right)\partial_x|p_\eps|
 + \sigma^2\left(\frac{\gamma'(w_\eps)}{\eps}+\frac{\gamma''(w_\eps)}{\gamma'(w_\eps)}\right) p_\eps \cdot\partial_x |p_\eps|. 
\end{multline*}
Noting that the expression on the first line is a third-degree polynomial in $\vert p_\eps\vert$, with coefficients either of order 1 or $\eps^{-1}$ and depending on $T$, there exists a large enough positive constant $\theta(T)$ such that
\begin{multline*}
\partial_t|p_\eps|-\frac{\sigma^2}{2}\partial^2_{xx}|p_\eps|
\leq -\frac{\sigma^2}{2}\left(|p_\eps|-\theta(T)\right)^3(1+\eps^{-1})
\\
-  \left( f(x)- \eps^{-1}\alpha(x) I_{\mu_\eps}(t)\right)\partial_x|p_\eps|
 + \sigma^2\left(\frac{\gamma'(w_\eps)}{\eps}+\frac{\gamma''(w_\eps)}{\gamma'(w_\eps)}\right) p_\eps \cdot\partial_x |p_\eps|.
\end{multline*}
Noticing that the spatially homogeneous function $y_\eps(t) = \sqrt{\frac{\eps}{\sigma^2(1+\eps)t}}+\theta(T),$ which is such that $y(0)=\infty$, is a solution of the equation:
\begin{multline*}
\partial_ty(t,x)-\frac{\sigma^2}{2}\partial^2_{xx}y(t,x)
= -\frac{\sigma^2}{2}\left(y(t,x)-\theta(T)\right)^3(1+\eps^{-1})
\\
-  \left( f(x)- \eps^{-1}\alpha(x) I_{\mu_\eps}(t)\right)\partial_xy(t,x)
 + \sigma^2\left(\frac{\gamma'(w_\eps)}{\eps}+\frac{\gamma''(w_\eps)}{\gamma'(w_\eps)}\right) y(t,x) \cdot\partial_x y(t,x),
\end{multline*}
we conclude that 
\[
|p_\eps(t,x)|\leq  \sqrt{\frac{\eps}{\sigma^2(1+\eps)t}}+\theta(T),\qquad 0<t\leq T. 
\]
\end{proof}

\begin{corollary}
\label{cor:regularity_hc}
Under the hypotheses of Proposition~\ref{prop:regularity_space}, the family of Hopf-Cole transformations $\{\phi_\eps\}_{\eps\leq 1}$ is uniformly bounded in compact subsets of $(0,\infty)$ $\times\R$.
\end{corollary}

\begin{proof}
Proposition~\ref{prop:supersol} implies that
$$
\phi_\eps(t,x)\leq -A'I_\eps(t)\Lambda(x)-\frac{\eps}{2}B'x^2+D't+E' \leq -A'I_\eps(t)\Lambda(x)+C(T),
$$
with $C(T)=D'T+E'$. Under the current assumptions, for $x>M$ for some $M>1$, we have:
%$$
%|\alpha(x)-Cx|\leq \delta\quad\Rightarrow\quad Cx-\delta\leq \alpha(x)\leq Cx+\delta
%$$
$$
Cx-\delta\leq \alpha(x)\leq Cx+\delta
$$
and thus
\begin{align*}
\Lambda(x) & = \int_0^M\alpha(s)ds +\int_M^x\alpha(s)ds \geq C(M) + \frac{Cx^2}{2}-\delta x
\\
&\geq C(M)+ \frac{Cx^2}{2}-\delta x^2 =: C(M)+C(\delta)x^2,
\end{align*}
with $C(\delta)$ positive for $M$ large enough; an analogous bound exists for $x\in(-\infty,-M]$. Since $I_\eps(t)$ is always positive, we have that
$$
\phi_\eps(t,x) \leq -AI_\eps(t)(C(M)+C(\delta)x^2)+C(T),
$$
and using that $I_\eps(t)\geq K^{-1}$ we finally have the desired quadratic bound:
$$
\phi_\eps(t,x) \leq -AK^{-1}C(\delta)x^2+C(M,T)=: -C_1x^2+C_2,
$$
which is only valid for $\vert x\vert >M$. Take now $R>M$ it follows that
$$
\int_{\R\setminus [-R,R]} \mu_\eps(t,x)\,dx =\int_{\R\setminus [-R,R]}e^{\phi_\eps(t,x)/\eps}\,dx \leq \int_{\R\setminus [-R,R]}e^{-(C_1x^2-C_2)/\eps}\,dx,
$$
%\textcolor{blue}{\sout{where both $C_1$ and $C_2$ do not depend on $R$. If we take $R$ large enough, then the exponent $-(C_1x^2-C_2)/\eps$ becomes strictly negative on the whole set $\R\setminus[-R,R]$ and then there is some $\eps_0$ small such that for all $\eps\leq \eps_0$:}
which, for $R$ large enough and $\eps\leq \eps_0$ for some $\eps_0$ small enough, is certainly upperbounded by $\frac 1 2$.
%$$
%\int_{\R\setminus [-R,R]} \mu_\eps(t,x)\,dx\leq \int_{\R\setminus [-R,R]}e^{-(C_1x^2-C_2)/\eps_0}\,dx<\frac12,
%$$
%and 
Mass conservation thus necessarily implies that $\int_{-R}^{R}e^{\phi_\eps(t,x)/\eps}\,dx\geq \frac12$. We deduce that for each $\eps\leq \eps_0$ there is some $x_\eps\in[-R,R]$ such that $\phi_\eps(x_{\eps})>-1$, meaning that:
$$
w_\eps(t,x_\eps)<\sqrt{2F^2+1}.
$$
In Proposition~\ref{prop:regularity_space} we showed that the space derivative of $w_\eps$ is uniformly bounded for each point in $[t,T]\times [-R,R]$ and then
$$
|w_\eps(t,x+h)-w_\eps(t,x)|\leq \left(\sqrt{\frac{1}{t\sigma^2}}+\theta(T)\right) h,
$$
and using $x_\eps$ and this previous inequality, we have that
$$
w_\eps(t,x) \leq w_\eps(t,x_\eps)+|w_\eps(t,x)-w_\eps(t,x_\eps)|\leq \sqrt{2F^2+1}+\left(\sqrt{\frac{1}{t\sigma^2}}+\theta(T)\right)R.
$$
By definition of $w_\eps$, we finally get the desired lower bound
$$
\phi_\eps(t',x) = 2F^2-w_\eps^2(t,x)\geq 2F^2-\left(\sqrt{2F^2+1}+\left(\sqrt{\frac{1}{t\sigma^2}}+\theta(T)\right)R\right)^2
$$
for all $(t',x)\in [t,T]\times[-R,R]$.
\end{proof}

\begin{remark}
\label{rem:unifcont}
The solution to the equation~\eqref{eq:hopfcole} is actually, uniformly continuous in space on $[t_0,T]\times[-R,R]$ for any $t_0>0$ and $R>0$ fixed. Indeed, notice that
$$
\phi_\eps(t,x)=2F^2-w_\eps(t,x)^2,
$$
therefore
$$
\partial_x\phi_\eps(t,x) = -2w_\eps(t,x)\partial_xw_\eps(t,x).
$$
From Proposition~\ref{prop:regularity_space} we have that the space derivative of $w_\eps$ is bounded, then we have
$$
|\partial_x\phi_\eps(t,x) | \leq 2|w_\eps(t,x)|\left(\sqrt{\frac{1}{t\sigma^2}}+\theta(T)\right),
$$
where $\theta(T)$ does not depend on $\eps$. From Corollary~\ref{cor:regularity_hc} we also have that $w_\eps$ is also bounded, thus
$$
|\partial_x\phi_\eps(t,x) | \leq 2\left(\sqrt{2F^2+1}+\left(\sqrt{\frac{1}{t\sigma^2}}+\theta(T)\right)R\right)\left(\sqrt{\frac{1}{t\sigma^2}}+\theta(T)\right):=M(t_0,T,\sigma),
$$
with $F$ not depending on $\eps$. As a consequence, for $t\in[t_0,T]$ and $x\in(-R,R)$ and $h$ small, we have
$$
|\phi_\eps(t,x+h)-\phi_\eps(t,x)|\leq M(t_0,T,\sigma)|h|.
$$
\end{remark}

The last auxiliary result we need to construct the proof of convergence of the families of Hopf-Cole transformations $\phi_\eps$ is the regularity of the solutions in time. The proof of the following result is an adaptation of the method used in~\cite[Lemma 9.1]{barles2002geometrical} (see also~\cite{perthame2008dirac,barles2009concentration} for other applications of these methods) and that we have also previously adapted to FitzHugh-Nagumo models with diffusive coupling in~\cite{quininao2020clamping}. 

\begin{lemma}
\label{lem:regularity_time}
For all $\eta>0$, $R>0$ and $t_0>0$, there exists a positive constant $\Theta$ such that for all $\eps<1$ and $(t,s,x)\in[t_0,T]\times [t_0,T]\times (-R/2,R/2)$ such that $0<t-s<\Theta$ we have
$$
|\phi_\eps(t,x)-\phi_\eps(s,x)|\leq 2\eta.
$$
\end{lemma}

\begin{proof}
The idea of the proof consists in finding two constants $A,B$ such that for any $x\in(-R/2,R/2)$, $s\in[t_0,T]$ and for every $\eps<\eps_0$:
$$
\phi_\eps(t,x)-\phi_\eps(s,x)\leq \eta+A(e^{\eps(x-y)^2}-1)+B(t-s),
$$
and
$$
-\eta-A(e^{\eps(x-y)^2}-1)-B(t-s)\leq \phi_\eps(t,x)-\phi_\eps(s,x).
$$
We present the proof of the first inequality; the second inequality is proven following the same arguments, and is thus omitted here. Fix some $s\in[t_0,T]$ and some $-\frac R2<x<\frac R2$, and define
 $$
 \varphi_\eps(t,y) = \phi_\eps(s,x)+\eta+A(e^{\eps(x-y)^2}-1)+B(t-s),
 $$
 with $K$ and $C$ constants that we will specify later on. The function $\varphi_\eps$ is clearly well defined for $t\in[s,T)$ and $y\in(-R,R)$. In Corollary~\ref{cor:regularity_hc} we showed that the function $\phi_\eps$ is bounded in compact sets of $(0,\infty)\times\R$, then
 $$
 M:=\|\phi_\eps\|_{L^\infty([t_0,T]\times[-R,R])}<+\infty.
 $$
In particular, we can find $A$ such that
$$
0<\frac{2}{e^{\frac{R^2\eps}{4}} -1}\|\phi_\eps\|_{L^\infty([t_0,T]\times[-R,R])}<A,
$$
and then taking $y\in\{-R,R\}$ we have the inequality
\begin{align*}
\varphi_\eps(t,\pm R) &=  \phi_\eps(s,x)+\eta+A(e^{\eps(x\mp R)^2}-1) +C(t-s)
\\
& \geq  \phi_\eps(t,\pm R) - 2\|\phi_\eps\|_{L^\infty}+\eta
+A(e^{\frac{R^2\eps}{4}} -1)
\\
%& \geq  \phi_\eps(t,R) - 2\|\phi_\eps\|_{L^\infty}+\eta+\frac{R^2}{4}K +C(t-s)
%\\
%& \geq  \phi_\eps(t,\pm R) - 2\|\phi_\eps\|_{L^\infty}+\eta+\frac{R^2}{4}\frac{8}{R^2}\|\phi_\eps\|_{L^\infty([t_0,T]\times[-R,R])} 
%\\
& \geq  \phi_\eps(t,\pm R) +\eta > \phi_{\eps}(t,\pm R),
\end{align*}
because $(x\mp R)^2\geq\frac{R^2}{4}$. This inequality extends to the whole interval $(-R,R)$. Indeed, if that were not the case, then there would exist $\eta>0$ such that for all constants $A$ we have some $y_{A,\eps}\in(-R,R)$ such that
$$
\varphi_\eps(s,y_{A,\eps}) = \phi_\eps(s,x)+\eta+A(e^{(x-y_{A,\eps})^2}-1) \leq \phi_\eps(s,y_{A,\eps})
$$
in particular
$$ \eta+A(e^{(x-y_{A,\eps})^2}-1)\leq \phi_\eps(s,y_{A,\eps})-\phi_\eps(s,x)\quad\Rightarrow\quad e^{(x-y_{A,\eps})^2}-1\leq 2MA^{-1},
$$
taking $A\rightarrow\infty$, we get that necessarily $y_{A,\eps}\rightarrow x$. Moreover, from remark~\ref{rem:unifcont} we know that $\phi_\eps$ is uniformly continuous in space on $[t_0,T]\times[-R,R]$, then for $A$ large enough
$$
|\phi_\eps(t,y_{A,\eps})-\phi_\eps(t,x)|\leq \frac{\eta}2,
$$
which is a contradiction. Finally, notice that
$$
\partial_t\varphi_\eps(t,y)=B,\quad \partial_y\varphi_\eps(t,y)=2A\eps(y-x)e^{\eps(y-x)^2},$$
and
$$
\partial_{yy}\varphi_\eps(t,y)=2A\eps e^{\eps(y-x)^2}+4A\eps^2(y-x)^2e^{\eps(y-x)^2},
$$
thus
\begin{multline*}
\partial_t\varphi_\eps +\left(\eps f'(y) - \alpha'(y)I_{\mu_\eps}(t) \right)
+  \left(\eps f(y)- \alpha(y) I_{\mu_\eps}(t)\right)\frac{1}{\eps}\partial_y\varphi - \frac{\sigma^2}{2\eps} \vert \partial_y\varphi \vert^2  - \frac {\sigma^2}{2} \partial_{yy}^2 \varphi
\\
=
B+\left(\eps f'(x) - \alpha'(x)I_{\mu_\eps}(t) \right)
-  \left(\eps f(y)- \alpha(y) I_{\mu_\eps}(t)\right)2A(y-x)e^{\eps(y-x)^2}
\\
- \frac{\sigma^2\eps}{2} \left(2A(y-x)e^{\eps(y-x)^2}\right)^2
-\frac{\sigma^2}{2}\left(2A\eps e^{\eps(y-x)^2}+4A\eps^2(y-x)^2e^{\eps(y-x)^2}\right),
\end{multline*}
all functions are bounded below, independently of $\eps$ as soon as it is smaller than 1, thus there are some constants $B, A$ such that
\begin{multline*}
\partial_t\varphi_\eps +\left(\eps f'(y) - \alpha'(y)I_{\mu_\eps}(t) \right)
\\
+  \left(\eps f(y)- \alpha(y) I_{\mu_\eps}(t)\right)\frac{1}{\eps}\partial_y\varphi - \frac{\sigma^2}{2\eps} \vert \partial_y\varphi \vert^2  - \frac {\sigma^2}{2} \partial_{yy}^2 \varphi
\\
B-C(C_f,C_\alpha,C_I)\,\geq 0,
\end{multline*}
whenever $B$ is large enough. From here, we deduce that $\varphi_\eps$ is a super-solution of the equation solved by $\phi_\eps$ and then
$$
\phi_\eps(t,y) \leq \varphi(t,y)=\phi_\eps(s,x)+\eta+A(e^{\eps(x-y)^2}-1) +B(t-s),
$$
the other inequality follows the same arguments. Finally, by taking  $y=x$ we have
$$
|\phi_\eps(t,x)-\phi_\eps(s,x)|\leq \eta+B(t-s)\leq \eta+B\Theta=2\eta,
$$
when $\Theta=\eta/B$.
\end{proof}

\section{The large coupling limit}
\label{sec:limit}

Equipped with all previous results, we can now take the limit $\eps\rightarrow0$ in the sequence $\phi_\eps$. This further allows us to characterize any adherence point of the set of Hopf-Cole transformations. Again, the technique adapts the methodology of~\cite{barles2009concentration} and references therein. In the following, we work on the compact subset $C_R=[t_0,T]\times [-R,R]$.

\begin{itemize} 
\item \textbf{Step 0:} From Lemma~\ref{lem:Jeps} we have that the integral part of the equation $I_\eps(t)$, after extraction of a subsequence, converges a.e. to a function $I(t)$ as $\eps$ goes to zero.

\item \textbf{Step 1:} so far we have that the sequence $\phi_\eps$ is uniformly bounded and uniformly continuous on $C_R$. The Arzela-Ascoli theorem implies that we can extract a subsequence $\phi_{\eps_n}(t,x)$ that converges locally uniformly to a continuous function $\phi(t,x)$.

\item \textbf{Step 2:}  The limit $\phi(t,x)$ is everywhere non-positive. Indeed, if there was a point $(t^*,x^*)\in[t_0,T]\times[-R,R]$ with $\phi(t^*,x^*)>0$, then there would be a small spatial interval $S$ such that for any $x\in S $ we would have that $\phi(t^*,x)>\delta$, for some $\delta>0$. Thanks to the positivity and mass conservation principle:
$$
1\geq \int_{S} e^{\phi_{\eps_n}(t,y)/\eps_n} dy\geq  |S| e^{\delta/\eps_n},
$$
which is not possible for small values of $\eps_n$. 

Moreover, the maximum of $\phi(t,x)$ is exactly $0$. Indeed, the quadratic upperbound we proved on $\phi_{\eps_n}$ for $R$ large enough,
$$
\phi_{\eps_n}(t,x)\leq -AK^{-1}C(\delta)x^2+C(M,T)=:C_1x^2+C_2,
$$
implies that $\phi_{\eps_n}$ is strictly negative for $R>0$ large enough. This implies that there is a value $R$ such that the limit $\phi(t,x)$ is strictly negative on $\R\setminus[-R,R]$. We thus deduce that
$$
1 = \lim_{n\rightarrow\infty}\int_{-R}^R e^{\phi_{\eps_n}(t,x)/\eps_n}\,dx,
$$
which implies that the maximum value of $\phi(t,x)$ cannot be strictly negative. As a result, the maximum value of $\phi(t,x)$ is necessarily zero.

\item \textbf{Step 3:} Fix some $t\in (0,1)$ and assume that for some $x\in\R$ the limit $\phi(t,x)=-a<0$. From uniform continuity in compacts of $C_R$, we find $\delta>0$ such that if $n$ is large
$$
\phi_{\eps_n}(t',x')\leq -\frac a2<0\text{ for all }(t',x')\in (t-\delta,t+\delta)\times (x-\delta,x+\delta),
$$
so
$$
\int_{x-\delta}^{x+\delta} \mu_{\eps_n}(t',x')\,dx'=\int_{x-\delta}^{x+\delta} e^{\phi_{\eps_n}(t',x')/\eps_n}\,dx'\xrightarrow{n\rightarrow\infty}0,
$$
therefore along converging subsequence the support of the limit $\mu_{\eps_n}(t,x)$ reduces to the set of $x$ such that $\phi(t,x)$ is zero, i.e.
$$
\text{sup}\,\mu(t,\cdot)\subset \{\phi(t,\cdot)=0\},\quad\text{for almost every $t$.}
$$
\item \textbf{Step 4:} For any $n$ we rearrange the terms of the equation~\eqref{eq:hopfcole} to find that $\phi_{\eps_n}$ is a solution to
\begin{multline*}
\eps_n\partial_t\phi_{\eps_n} = -\eps_n\left(\eps_n f'(x) - \alpha'(x)I_{\mu_{\eps_n}}(t) \right)
\\
-  \left(\eps_n f(x)- \alpha(x) I_{\mu_{\eps_n
}}(t)\right)\partial_x\phi_{\eps_n} + \frac{\sigma^2}{2} \vert \partial_x \phi_{\eps_n} \vert^2  + \eps_n\frac {\sigma^2}{2} \partial_{xx}^2 \phi_{\eps_n}.
\end{multline*}
Define $H_\eps\in C(\R\times \R)$ by
\begin{align*}
H_\eps(x,p) =& -\eps\left(\eps f'(x) - \alpha'(x)I_{\mu_{\eps}}(t) \right)
-  \left(\eps f(x)- \alpha(x) I_{\mu_{\eps
}}(t)\right)p+ \frac{\sigma^2}{2} p^2,
\end{align*}
which as $\eps$ vanishes along the sequence $(\eps_n)$, converges locally uniformly toward
$$
H(x,p) = \alpha(x)I(t)\,p+\frac{\sigma^2}{2}p^2,
$$
since $I_{\mu_{\eps_n}}(t)$ converges uniformly toward $I(t)$ on any interval $[0,T]$ with $T>0$. At the same time, $\phi_{\eps_n}$ converges locally uniformly on $(0,+\infty)\times\R$ toward the continuous function $\phi(t,x)$ (i.e. in compact subsets).

Let $\varphi\in C_b^2(\R_+\times\R)$ be a regular test function with $\varphi(t',x')=\phi(t',x')$ for some $(t',x')$, such that $\phi(t,x)\leq\varphi(t,x)$ at a small neighborhood of $(t',x')$. Similarly, for each $\eps_n$, take a test function $\varphi_{n}$, a time $t'_n$ and a point $x_n'$ satisfying the same properties. It holds that
\begin{multline*}
\eps_n\partial_t\varphi_{\eps_n} +  \left(\eps_n f(x)- \alpha(x) I_{\mu_{\eps_n}}(t)\right)\partial_x\varphi_{\eps_n} - \frac{\sigma^2}{2} \vert \partial_x \varphi_{\eps_n} \vert^2  - \eps_n\frac {\sigma^2}{2} \partial_{xx}^2 \varphi_{\eps_n}
\\
\leq -\eps_n\left(\eps_n f'(x) - \alpha'(x)I_{\mu_{\eps_n}}(t) \right),
\end{multline*}
but $I_{\mu_{\eps_n}}(t)\rightarrow I(t)$ a.e. $t\geq0$ as $n\rightarrow\infty$. Using the local uniform convergence, it follows that
$$
-\alpha(x')I(t')\partial_x\varphi(t',x')-\frac{\sigma^2}{2}|\partial_x\varphi(t',x')|^2\leq 0,
$$
therefore $\phi(t,x)$ is a subsolution to the equation
\begin{equation}
\label{eq:limit_eq}
-\alpha(x) I(t)\, \partial_x\phi(t,x)-\frac{\sigma^2}{2}|\partial_x\phi(t,x)|^2= 0.
\end{equation}
By a similar argument, we can prove that $\phi(t,x)$ is also a supersolution to the previous equation, and in consequence $\phi(t,x)$ is a viscosity solution to~\eqref{eq:limit_eq}
\end{itemize}

\end{document}